\documentclass[11pt]{article}
\usepackage[affil-it]{authblk}
\usepackage{geometry}                
\usepackage[mathscr]{euscript}
\usepackage{epigraph}
\usepackage{graphicx}
\usepackage{color}
\usepackage{subfig}
\usepackage{enumitem} 
\setlist[enumerate]{leftmargin=1.5cm,rightmargin=0.5cm,noitemsep, topsep=2pt}
\setlist[itemize]{leftmargin=1.5cm,rightmargin=0.5cm,noitemsep, topsep=2pt}

\usepackage{algorithm}
\usepackage[noend]{algpseudocode}

\usepackage{mdframed}

\definecolor{clemson-orange}{RGB}{234,106,32}
\definecolor{chicago-maroon}{RGB}{128,0,0}
\definecolor{northwestern-purple}{RGB}{82,0,99}
\definecolor{cornell-red}{RGB}{179,27,27}
\definecolor{sauder-green}{RGB}{171,180,0}
\definecolor{lawngreen}{RGB}{0,250,154}

\usepackage[numbers,sort&compress]{natbib}

\makeatletter
\def\BState{\State\hskip-\ALG@thistlm}
\makeatother

\usepackage{amsmath}
\usepackage{amssymb}
\usepackage{amsthm}
\allowdisplaybreaks

\newcommand{\bb}{\mathbb}

\newcommand{\N}{{\bb N}}
\newcommand{\Q}{{\bb Q}}
\newcommand{\R}{\bb R}

\theoremstyle{definition}
\newtheorem{theorem}{Theorem}
\newtheorem{lemma}{Lemma}
\newtheorem{corollary}{Corollary}
\newtheorem{proposition}{Proposition}
\newtheorem{assumption}{Assumption}

\newtheorem{remark}{Remark}
\newtheorem{example}{Example}

\usepackage[colorlinks,citecolor=northwestern-purple,urlcolor=chicago-maroon,linkcolor=chicago-maroon,backref=page]{hyperref}

\usepackage[nameinlink]{cleveref}
\crefname{assumption}{Assumption}{Assumptions}
\crefname{lemma}{Lemma}{Lemmas}
\crefname{theorem}{Theorem}{Theorems}
\crefname{corollary}{Corollary}{Corollaries}
\crefname{proposition}{Proposition}{Propositions}
\crefname{claim}{Claim}{Claims}
\crefname{algorithm}{Algorithm}{Algorithms}
\crefname{figure}{Figure}{Figures}
\crefname{remark}{Remark}{Remarks}
\crefname{section}{Section}{Sections}
\crefname{appendix}{Appendix}{Appendices}
\crefname{procedure}{Procedure}{Procedures}
\crefname{example}{Example}{Examples}
\crefname{table}{Table}{Tables}
\crefname{equation}{}{}
\crefname{enumi}{}{}

\title{A geometric simplex method in infinite-dimensional spaces}

\author{
Robert L. Smith\thanks{Industrial and Operations Engineering, University of Michigan, E-mail: rlsmith@umich.edu}\qquad   
Christopher Thomas Ryan%
\thanks{UBC Sauder School of Business, University of British Columbia, E-mail: chris.ryan@sauder.ubc.ca}
}
\affil{}

\begin{document}
%%%%%%%%%%%%%%%%

\maketitle

\begin{abstract}
We expand the basic geometric elements of the simplex method to linear programs in locally convex topological vector spaces and provide conditions under which the method converges in value to optimality. This setting generalizes many previous investigations of the simplex method, which are restricted to Hilbert spaces or otherwise specially structured instances. Our generality is obtained by avoiding the ``algebraic'' machinery of pivoting via column operations, which has required stronger topological conditions in establishing a connection between basic feasible solutions and extreme point structure. We show that our definition of polytopes captures optimization over the Hilbert cube, a quintessential object in infinite-dimensional spaces known for its surprisingly complicated properties. Moreover, all polytopes (under our definition) have exposed extreme points connected by edge paths. 
\end{abstract}

\section{Introduction}\label{s:intro}

Infinite-dimensional linear programming is both classical and cutting-edge. Some of the very first linear programs studied were stated over infinite-dimensional spaces \cite{monge1781memoire,kantorovitch1958translocation}, and there has been recent growing interest in several applied areas of mathematics that concern infinite-dimensional linear programs \cite{dworczak2024persuasion,kleiner2019strong,guo2022online,ghate2021simplex,ryan2018simplex,hernandez2012discrete,chand2002forecast,kryazhimskii2001constraint}.

Since at least \cite{anderson-nash}, researchers have been interested in extensions of the simplex method to understand the structure and optimality of linear programs in topological vector spaces. Anderson and Nash \cite{anderson-nash} were interested in generalizing the \emph{algebraic structure} underlying the simplex method, including notions of column bases, basic feasible solutions, and reduced costs. This program was championed over the years by several research groups 
(for example, \cite{pullan1993algorithm,ferris1989interior,smith1998infinite,ghate2013markov,sharkey2008simplex,shindin2020simplex,weiss2008simplex,anderson2001simplex,anderson1989extension}), including recent advancements \cite{ghate2021simplex,shindin2020simplex}. In this literature, the focus is largely on finding special types of linear programs where these algebraic concepts prove meaningful and amenable to analysis. In general, the dictionary between extreme points and basic feasible solutions that is foundational to the simplex method in finite dimensions (see Chapter 2 of \cite{bertsimas1997introduction}) breaks down in topological vector spaces (as detailed at length in \cite{anderson-nash}). 

This paper pursues a different program. We are interested in generalizing the classical \emph{geometric} understanding of the simplex method: start with an extreme point and iteratively follow the edges to adjacent extreme points until you reach an extreme point from which no adjacent edge direction improves the objective. We do not concern ourselves with finding an underlying algebraic description of extreme points. Despite alleviating ourselves of this concern, this new path is still fraught with complexities. It is not immediately obvious how to (i) define ``edges'', (ii) ensure that moving along the edges yields another extreme point, and (iii)  that the process makes guaranteed progress toward optimality. We tackle these and other issues in this paper. 

Our main result is to produce a geometric simplex algorithm with the following properties. Starting with an extreme point, it can identify a set of adjacent extreme points and define edges on which to reach them. We show (\cref{theorem:convergence-in-value}) that this algorithm converges to optimal value (that is, the objective values of successive iterates converge to the optimal value of the program). When a unique optimal solution exists, the iterates of the algorithm converge to the optimal solution (\cref{cor:convergence-in-solution}). Though the goal seems simple, a carefully constructed set of assumptions was required (see \cref{ass:compact}--\cref{ass:uniform-convergence-edge-costs} throughout the paper). These assumptions will be motivated and justified as they are presented and we will not attempt to recapitulate that careful development here. A summarizing discussion of the assumptions can be found in \cref{s:conclusions}.

Our algorithm has important limitations. We do not guarantee that each iteration can be executed in finite time, let alone that an optimal solution can be found in finite time. However, it ensures that the simplex method has the necessary geometric structure of extreme points, edges, and improving directions so that the algorithm does not get ``caught'' in locally optimal solutions or be unable to determine where to go next. The assumptions all drive toward validating the basic mechanism of ``moving from extreme point to extreme point while improving the objective'' under minimal restrictiveness.
Our goal is more conceptual than providing an implementable algorithm. We are interested in the general geometric structure of a compact convex set described by linear inequalities that make it amenable to a simplex method. 

\medskip
Our assumptions are broad enough to capture most of the well-known examples of when a simplex method ``works'' in the sense we put forward here, including the (corrected)%
\footnote{We say ``corrected'' conditions because we found an error in the development of \cite{ghate2021simplex}. We propose a fix to recover the results of that paper in \cref{s:fix-ghate}.}
conditions of \cite{ghate2021simplex} that define a broad class of countably-infinite linear programs (including many instances of previous investigations). Our approach allows us to move beyond the sequence space setting of \cite{ghate2021simplex} and related papers, opening the door to problems defined over non-sequence spaces, including function and measure spaces. 

However, we show that a very classical object --- the Hilbert cube --- fails the conditions of \cite{ghate2021simplex} but nonetheless satisfies our assumptions (see \cref{ex:hilbert-cube,ex:hilbert-cube-ass:uniform-convergence-edge-costs,ex:hilbert-cube-bounded-at-extremes,ex:hilbert-cube-partial-sums-compact,ex:hilbert-cube-steepest-exists,ex:hilbert-cube-edge-lengths-bounded-away-from-zero,ex:hilbert-cube-countably-many-constraints}).  This demonstrates that our simplex method is broader than that explored in the previous literature. The fact that the Hilbert cube is problematic for previous methods is not too surprising, despite its simplicity, as it is a notoriously challenging example to understand and analyze \cite{armstrong1977polyhedrality}.

Our assumptions provide a novel definition of what we will call a \emph{polytope} for general topological vector spaces: a polytope is an object where a geometric simplex method can be used to optimize a linear functional over it, potentially requiring infinitely many iterations, under our assumptions. Why might there be interest in such a definition? The familiar definitions of a polytope do not extend to general topological spaces as we might hope. In finite-dimensional space, there are two alternative definitions of a polytope: (i) the convex hull of finitely many points or (ii) the bounded intersection of finitely many halfspaces. Both of these conceptions are flawed in infinite-dimensional space. For (i), the convex hull of finitely many points is only finite-dimensional, a trivial restriction in infinite-dimensional spaces. Indeed, several papers have successively executed simplex-like algorithms on infinite-dimensional objects \cite{sharkey2008simplex,ghate2013markov,ryan2018simplex,ghate2021simplex}. The problem with (ii) is arguably even more severe: in infinite-dimensional space, the intersection of finitely many halfspaces is never bounded.

Other authors have proposed alternative definitions for polytopes, while there is generally no consensus on how they should be defined. In our search of the literature, we found at least twelve alternative definitions of ``polytope'' in general spaces (and make no claims of exhaustion) \cite{bastiani1959cones,klee1956strict,maserick1965convex,alfsen1969vertices,phelps1969infinite,amir1972suns,gleit1972note,rajagopalan1974maximal,lau1974infinite,brosowski1974some,durier1986sets}. For a comprehensive overview of this topic and related themes see  \cite[Chapter 6]{fonf2001infinite}. Several papers have attempted to reconcile this multitude of definitions, including \cite{armstrong1977polyhedrality,durier1993polyhedral,fonf2004infinite}. \cref{s:defining-polytopes} provides additional discussion and shows that the Hilbert cube, arguably the simplest meaningful feasible region one can imagine, fails to meet the definition of many current notions of a polytope in infinite dimensions but nonetheless satisfies our assumptions.

The rest of the paper is organized as follows. In \cref{s:setup}, we provide the details of our basic settings and introduce the main notation. There, we introduce our object of study, which we denote by $\mathcal P$, the intersection of halfspaces in a topological vector space. \cref{s:geometric-properties} elaborates the geometric properties of $\mathcal P$, including extreme points, edges, and support cones. \cref{s:geometric-simplex-method} presents our geometric simplex method building on these geometric structures. This section contains our main result (\cref{theorem:convergence-in-value}). Finally, in \cref{s:defining-polytopes}, we briefly discuss the challenge of defining polytopes in topological vector spaces and how our simplex method approach enters the conversation. \cref{s:geometric-simplex-method,s:defining-polytopes} both discuss the mysterious example of the Hilbert cube. \cref{s:conclusions} concludes, including some additional discussion of our assumptions. The proofs of propositions and theorems are found in the main text; the proofs of technical lemmas are placed in the appendix. 

\section{Setup}\label{s:setup}

We start with the basic setup of our problem and introduce all of the preliminary notation. For ease of reference, \cref{fig:basic-notions} captures much of the notation of this section in a finite-dimensional setting. 

Let $X$ be a real locally convex Hausdorff topological vector space \cite[p. 53]{holmes2012geometric} with topology $\tau$. Let $X^*$ be the set of continuous linear functionals on $X$; that is, functions $\phi : X \to \R$ that are continuous in the $\tau$ topology of $X$ and the usual topology on the reals $\R$. A proper affine subspace that is maximal (when ordered by inclusion) is a {\em hyperplane} in $X$. From \cite[p. 3]{holmes2012geometric}, an affine subspace $H$ is a closed hyperplane if and only if there is some non-zero continuous linear functional $\phi\in X^*$ and scalar $b\in \R $ such that $H=\{x\in X \mid \phi(x) = b\}$. We call $S = \{x \in X \mid \phi(x) \le b\}$ the \emph{halfspace} associated with hyperplane $H$. Since $\phi \in X^*$, $S$ is closed (in the topology $\tau$).

Given an index set $A$, for all $\alpha \in A$ define a hyperplane $H_\alpha \triangleq \{x\in X \mid \phi_\alpha(x)=b_\alpha\}$ where $\phi_\alpha\in X^*$ and $b_\alpha\in \R$. Let $\{S_\alpha \mid \alpha\in A\}$ be the associated collection of halfspaces (called \emph{constraints}) of the form
\begin{equation*}
S_\alpha \triangleq \{x\in X \mid \phi_\alpha(x)\le b_\alpha\}.
\end{equation*}
Our main object of interest is the set 
\begin{equation}\label{eq:region-of-interest}
\mathcal P \triangleq \bigcap_{\alpha\in A} S_\alpha,
\end{equation}
which denotes the intersection of all halfspaces in the family $\{S_\alpha \mid \alpha\in A\}$.%
\footnote{We do not call $\mathcal P$ a polytope without further assumptions; this definition is more general than what we will call a polytope. Indeed, a disk in $\R^2$ satisfies this definition and is clearly not a polytope.} 
Clearly, $\mathcal P$ is a closed convex set, since each $S_\alpha$ is closed and convex, and these properties are conserved under (arbitrary) intersections. Conversely, every closed convex set in $X$ has the form \cref{eq:region-of-interest} for some index set $A$ and constraint set $\{S_\alpha\mid \alpha \in A\}$ \cite[Corollary 5.83]{hitchhiker}.   

We say a constraint $S_\alpha$ is \emph{active} (or \emph{tight}) at a point $p \in \mathcal P$ if $p \in H_\alpha$. For a given $p \in \mathcal P$, let $A(p)$ denote the set of active indices of constraints at $p$. The set 
\begin{equation}\label{eq:tight-constraints-at-p}
\mathcal H (p) \triangleq \bigcap_{\alpha \in A(p)} H_\alpha =  \{x\in X \mid \phi_\alpha(x)=b_\alpha, \alpha\in A(p)\}
\end{equation}
defines the intersection of all hyperplanes in the family $\{H_\alpha \mid \alpha \in A\}$ that pass through the point $p$. Note that $\mathcal H(p)=X$ when $A(p)=\emptyset$. 

At a point $p\in \mathcal P$, for all $\alpha \in A$ we let 
\begin{equation}\label{eq:slack}
s_\alpha (p) \triangleq b_\alpha - \phi_\alpha(p) 
\end{equation}
denote the {\em slack} of constraint $S_\alpha$ at the point $p$. Constraints that are tight at $p$ have a slack of $0$.

To avoid trivialities throughout, we assume $\mathcal P$ is not a singleton $\{p\}$.

\begin{figure}
    \centering
    \includegraphics[width=0.45\linewidth]{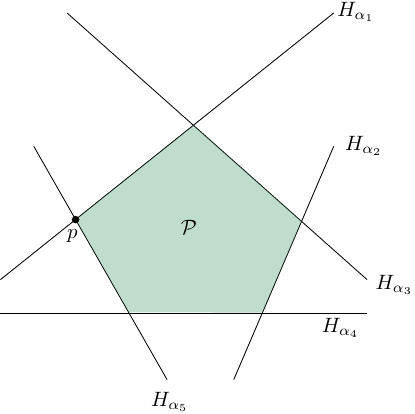}
    \caption{An illustration of the basic notions in the paper. Here $A = \{\alpha_1, \alpha_2, \dots \alpha_5\}$, $A(p) = \{\alpha_1, \alpha_5\}$, $\mathcal H(p) = H_{\alpha_1} \cap H_{\alpha_5} = \{p\}$, and $s_{\alpha_1}(p) = s_{\alpha_5}(p) = 0$ with $s_{\alpha_j}(p) > 0$ for $j = 2,3,4$.}
    \label{fig:basic-notions}
\end{figure}

\section{Geometric properties}\label{s:geometric-properties}

The geometry of the simplex method is built on the notions of extreme points and edges that connect adjacent extreme points. Readers who are familiar with these objects should take note of how our approach differs from the usual approaches in finite-dimensional spaces (like in \cite{bertsimas1997introduction}) by relying more exclusively on the algebraic structures of the underlying topological vector space $X$ than other treatments. 

\subsection{Extreme point structure}\label{ss:extreme-points}

Recall that an \emph{extreme point} of a set $K$ is an element $x \in K$ where there does not exist $y, z \in K$ such that $x \in (x',x'') \triangleq \{\lambda x' + (1-\lambda) x'' \mid \lambda \in (0,1) \}$ and $x' \neq x''$. Let $E$ be the set of all extreme points of $\mathcal P$.

Our first task is to relate the algebraic definition of extreme point to the halfspace structure that defines $\mathcal P$. This relationship is well-understood in finite-dimensional vector spaces $X$ \cite[Section~2.2]{bertsimas1997introduction}. The situation in more general $X$ is more complex (see \cite[Section~2.3]{anderson-nash} and \cite{ghate2009characterizing}). One direction of the relationship is relatively straightforward:

\begin{proposition}\label{prop:intersect-then-extreme}
A point $p\in \mathcal P$ is an extreme point of $\mathcal P$ if $\mathcal H(p)=\{p\}$. 
\end{proposition}
\begin{proof}
Suppose $\mathcal H(p)=\{p\}$ and, by way of contradiction, that $p=\lambda x' + (1-\lambda) x^{\prime\prime}$ for some $x', x'' \in \mathcal P$ and $\lambda \in (0,1)$ with $x^{\prime}\neq x^{\prime\prime}$. 

Let's derive a contradiction to $x' \neq x''$. We do so by showing that $x', x'' \in \mathcal H(p)$ and using the fact that $\mathcal H(p)$ is the singleton $\{p\}$ and so $x', x'' \in \mathcal H(p)$ implies that $x' = x'' = p$. 

It remains to show that $x', x'' \in \mathcal H(p)$. In particular, from \cref{eq:tight-constraints-at-p} it suffices to show that $x', x'' \in H_\alpha$ for all $\alpha \in A(p)$. To see this, select an arbitrary $\alpha \in A(p)$. If both $x', x'' \in H_\alpha$ then we are done, and so may assume that one of $x'$ or $x''$ is not in $H_\alpha$. This violates the fact that $\alpha \in A(p)$. Indeed, since both $x'$ and $x''$ are in $S_\alpha$ (due to the fact $x', x'' \in \mathcal P$), it must be that either $\phi_\alpha(x') < b_\alpha$ or $\phi_\alpha(x'') < b_\alpha$. But then from the linearity of $\phi_\alpha$ we have, $\phi_\alpha(p)=\lambda \phi_\alpha(x')+(1-\lambda) \phi_\alpha(x^{\prime\prime})<b_\alpha$. But this contradicts the fact that $\alpha \in A(p)$; i.e.,   $\phi_\alpha(p) = b_\alpha$. 

We conclude that $x', x'' \in H_\alpha$, completing the proof.
\end{proof}

This leaves open the question of whether $p$ being an extreme point implies $\mathcal H(p) = \{p\}$. The following example shows that this need not be the case. 

\begin{figure}
    \centering
    \includegraphics[width=0.2\linewidth]{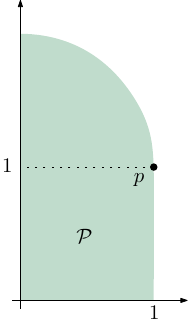}
    \caption{A counter-example to the converse of \cref{prop:intersect-then-extreme}.}
    \label{fig:non-exposed-face}
\end{figure}

\begin{example}\label{example:counter-example-to-converse}
Consider the set $\mathcal P$ shown in green in \cref{fig:non-exposed-face}. Since it is a closed convex set, we know $\mathcal P$ is of the form \cref{eq:region-of-interest}. However, notice that although $p$ is an extreme point of $\mathcal P$, we do not have $\mathcal H(p) = \{p\}$. Indeed, $\mathcal H(p)$ lies along the line segment from $(1,0)$ to $p$.%
\footnote{It is worth noting that $p$ is known as a non-exposed extreme point, see \cite[Chapter~12]{soltan2019lectures}.} $\blacktriangleleft$
\end{example}

To prove the converse, we make a few assumptions. 

\begin{assumption}[Compactness]\label{ass:compact}
The set $\mathcal P$ is non-empty and compact (in the topology $\tau$).
\end{assumption}

For ease of exposition, we will assume all extreme points are {\em non-degenerate}. That is, an extreme point $p\in E$ is not  the  unique intersection of any strict subset of its active constraints.
Also  we will  require there are no {\em phantom constraints} which arise implicitly as the limit of a  sequence of  constraints.  For example, consider in $\mathbb{R}^n$ the sequence of constraints $x_j\le 1+(1/2)^j, j =1,2,\ldots$ which is equivalent to the presence of a single constraint $x\le 1$. In, for example, a two dimensional box determined by phantom constraints, each extreme point  has {\em no} active constraints. We want conditions that exclude the case of either too many and too few active constraints at an extreme point. These non-degeneracy and no phantom constraints requirements  together that we will term {\em strict non-degeneracy} prevents over and under representation of active constraints in the determination of an extreme point. In \cref{ss:edge-structure}, we will use these assumptions to establish a one-to-one correspondence between active constraints at an extreme point and edges emanating from it. 

\begin{assumption}[Strict non-degeneracy]\label{ass:support-conditions}
\begin{itemize}
    \item[a)] For all extreme points $p\in E$ and all indices $\alpha \in A(p),$   $\cap_{\alpha^\prime\in A(p),\alpha^\prime\neq \alpha} H_{\alpha^\prime} \supset \{p\}$.
    \item[b)] Moreover, $\rho \triangleq \inf_{p\in E,\alpha\in A \setminus A(p)} s_\alpha(p)>0$.
\end{itemize}
\end{assumption}

\begin{assumption}[Uniformly bounded linear functionals]\label{ass:constraints-bounded}
For all $p \in E$ and all $r\in \mathcal H (p) - p$, $\sup_{\alpha\in A\setminus A(p)} \phi_\alpha (r) <\infty$.
\end{assumption}

\begin{example}[Hilbert cube]\label{ex:hilbert-cube}
Consider the following example inspired by the development in \cite{ghate2021simplex}. Let $X$ be the Hilbert space defined in that paper, namely
\[
X \;=\; \Bigl\{x=(x_j)_{j\ge 1} : \sum_{j=1}^\infty \delta_j^2 x_j^2 < \infty\Bigr\},
\qquad 
\langle x,y\rangle_X \;=\; \sum_{j=1}^\infty \delta_j^2 x_j y_j,\quad 
\|x\|_X^2=\sum_{j=1}^\infty \delta_j^2 x_j^2,
\]
with weights $(\delta_j)$ chosen so that $\delta_j\in(0,1)$ and $\sum_{j=1}^\infty \delta_j^2<\infty$ (for instance, $\delta_j=\delta^j$ with $0<\delta<1$). Define the “Hilbert cube”
\[
\mathscr H \;=\; \prod_{j=1}^\infty [0,1] \;\subseteq\; X.\footnote{For a similar definition of the Hilbert cube, see \cite[Section 3.9]{hitchhiker}.}
\]
We may express $\mathscr H$ in the form of \cref{eq:region-of-interest} by setting up the constraints to express $0 \le x_j \le 1$ for all $j = 1, 2, \dots$. We can capture this by setting $A = \{1,2,\dots\} \cup \{-1, -2, \dots\}$ where for $\alpha \in \{1,2,\dots\}$ we set $\phi_\alpha(x) = \langle x, y^\alpha \rangle$ where $y^\alpha_j = 1$ for $j = \alpha$ and $0$ otherwise with $b_\alpha = \delta_\alpha^2$. For $\alpha \in \{-1,-2, \dots\}$, set $\phi_\alpha(x) = \langle x, y^\alpha \rangle$ where $y^\alpha_j = -1$ for $j = \alpha$ and $0$ otherwise with $b_\alpha = 0$. 

We now show that the Hilbert cube satisfies \cref{ass:compact,ass:support-conditions,ass:constraints-bounded}. Compactness follows by applying Tychonoff's theorem, establishing \cref{ass:compact}. 

As for \cref{ass:support-conditions}, it is easy to see the extreme points of $H$ have the form $p$ such that $p_j = 0$ or $1$ for all $j = 1, 2, \dots$. Observe that for all $j$ such that $p_j = 0$, the constraint $-j \in A$ is tight and constraint $j \in A$ is slack with $s_j(p) = 1$. Also, for all $j$ such that $p_j = 1$ we have constraint $j \in A$ is tight and constraint $-j \in A$ is slack with $s_{-j}(p) = 1$. In other words, $\inf_{\alpha \in A \setminus A(p)} s_\alpha(p) = 1$ for all $p \in E$ and so in \cref{ass:support-conditions}, $\rho = 1 > 0$. This establishes \cref{ass:support-conditions}. 

\cref{ass:constraints-bounded} is nearly vacuous in this setting. It is easy to see that all extreme points $p$ have $\mathcal H(p) = \{p\}$ and so the only possible value for $r$ to consider in \cref{ass:constraints-bounded} is $r = 0$. $\blacktriangleleft$
\end{example}

The example shown in \cref{fig:non-exposed-face} fails \cref{ass:support-conditions} at $p$. Indeed, the constraints that define the ``quarter disk'' part of the figure about the line $y = 1$ have slacks that diminish to 0 at the point $p$. 

Under these assumptions, we can establish the converse of \cref{prop:intersect-then-extreme}. The argument requires the following two technical lemmas (with proofs relegated to the appendix). A \emph{line} refers to a one-dimensional affine subspace. We refer to $0$-dimensional affine subspaces (that is, singletons) as \emph{trivial} affine subspaces.

\begin{lemma}\label{lemma:0}
Suppose a line intersects a hyperplane but does not lie within that hyperplane. Then the line intersects the hyperplane at a single point.
\end{lemma}

\begin{lemma}\label{lemma:00}
If $S$ is a non-trivial affine subspace in $X$ and $H$ is a hyperplane intersecting $S$ at the singleton point $p$, then $S$ is a line through $p$.    
\end{lemma}

\begin{proposition}\label{prop:extreme-then-intersect}
If \cref{ass:compact,ass:support-conditions,ass:constraints-bounded} hold and $p$ is an extreme point of $\mathcal P$ then $\mathcal H (p)= \{p\}$.
\end{proposition}
\begin{proof}
We prove this by establishing the contrapositive. Suppose  $\mathcal H(p)\ne \{p\}$ and so there exists a $q\in \mathcal H (p)$ with $q\ne p$. Let $r=q-p$. The goal is to show that $p$ is not an extreme point.

The line $\ell(p,r) \triangleq \{p + \lambda r \mid \lambda\in \R\}$ lies in $\mathcal H (p)$ since $\mathcal H (p)$ is a non-trivial affine subspace containing $p$ and $q$.

Let $A^+ \subseteq A \setminus A(p)$ be the set of indices of hyperplanes $H_\alpha$ (not in $A(p)$) that intersect $\ell(p)$ at $z_\alpha = \lambda_\alpha r + p$ for $\lambda > 0$. The set $A^-$ is defined similarly but for $\lambda < 0$. 

Note that $A^+$ and $A^-$ may be finite or infinite but both must be nonempty. This follows by \cref{ass:compact} and \cite[p. 736]{klee1956strict} that states that a compact convex set (namely, $\mathcal P$) in a linear topological space is linearly bounded, i.e. its intersection with any line is bounded. However, if (say) $A^+$ were empty then the ray $\{p + \lambda r \mid \lambda \ge 0\}$ would lie entirely in $\mathcal P$, contradicting the fact that $\mathcal P \cap \ell(p)$ is bounded. 

We argue that this allows us to move from $p$ in both directions along $\ell(p,r)$ and still remain in the set $\mathcal P$. 

To do this observe that we have for $\alpha\in A^+ \cup A^-$, $b_\alpha = \phi_\alpha(z_\alpha)=\phi_\alpha(\lambda_\alpha r+p)$ and by linearity of $\phi_\alpha$,  $\lambda_\alpha=(b_\alpha-\phi_\alpha(p))/\phi_\alpha(r)=s_\alpha(p)/\phi_\alpha(r)$. Hence 
\begin{align*}
\inf_{\alpha\in A^+}\lambda_\alpha & =\inf_{\alpha\in A^+} \frac{s_\alpha(p)}{\phi_\alpha(r)}  \\
& \ge \inf_{\alpha\in A\setminus A(p)} \frac{s_\alpha(p)}{\phi_\alpha(r)}  \\
& \ge  \inf_{\alpha\in A\setminus A(p)} \frac{s_\alpha(p)}{\sup_{\alpha\in A\setminus A(p)} \phi_\alpha(r)}>0
\end{align*}
where the second inequality follows from \cref{ass:support-conditions,ass:constraints-bounded}. Similarly, $\inf_{\alpha\in A^-}-\lambda_\alpha >0$.

Therefore, $p$ is in the interior of the line segment joining $p^+=\{\inf_{\alpha \in A^+} \lambda_\alpha\} r + p$ and $p^-=\{-\inf_{\alpha\in  A^-} -\lambda_\alpha\} r + p$ which is in $\mathcal P$ because $p^+$ and $p^-$ are in $\mathcal P$ as they are in $\mathcal H(p)$ as well and thereby satisfy the constraints passing through $p$ and also satisfy  inequality constraints not passing through $p$. Therefore the line segment joining $p^+$ and $p^-$ is in $\mathcal P$ by the convexity of $\mathcal P$. We have $p$ is in the interior of this line segment in $\mathcal P$ and so $p$ is not an extreme point of $\mathcal P$. 
\end{proof}

From \cref{prop:extreme-then-intersect,prop:intersect-then-extreme}, we have the following result:

\begin{theorem}\label{thm:characterize-EP}
Suppose \cref{ass:compact,ass:support-conditions,ass:constraints-bounded} hold. A point $p \in \mathcal P$ is an extreme point of $\mathcal P$ if and only if the hyperplanes passing through $p$ intersect in the singleton $\{p\}$, i.e. $\mathcal H(p) = \{p\}$.    
\end{theorem}

\subsection{Edge structure}\label{ss:edge-structure}

% \chris{Put in non-degenerate assumption. We don't need non-redundancy. We can remove infinite-time preprocessing step.}

% In the development that follows we will work with a minimal set of constraints to define $\mathcal P$; that is, we will assume that the set of constraints $\{S_\alpha \mid \alpha \in A\}$ is such that for every $\alpha' \in A$, $\cap_{\alpha \in A \setminus \{\alpha'\}} S_\alpha \not\subset \cap_{\alpha \in A} S_\alpha$. We call such a set of constraints \emph{non-redundant}. The following result shows that we can assume this without loss using Zorn's lemma. 

% \begin{lemma}[Non-redundant]\label{lemma:nondegeneracy}
% Let $\{S_\alpha \mid \alpha \in A\}$ be a set of constraints such that $\mathcal P = \cap_{\alpha \in A} S_\alpha$. There exists a subset of constraints $A^* \subseteq A$ such that $\mathcal P = \cap_{\alpha \in A^*} S_\alpha$ and the set of constraints $\{S_\alpha \mid \alpha \in A^*\}$ is non-redundant. 
% \end{lemma}

\textcolor{blue}{Defining edges starts with defining the following lines emanating from extreme points. By \cref{ass:support-conditions}(a), for all extreme points $p\in E$ and  $\alpha\in A(p)$,  the set 
\begin{equation}\label{eq:emenating-line}
\ell_{\alpha} (p) = \cap_{\alpha'\in A(p),\alpha'\neq \alpha} H_{\alpha'}
\end{equation}
is a line in $X$ passing through $p$.}

In the remainder of this section, we give a geometric interpretation of the lines $\ell_\alpha(p)$ for $\alpha \in A(p)$ in relation to the set $\mathcal P$. As we shall see, the lines ``traces'' out an edge $e_\alpha(p)$ of $\mathcal P$ between $p$ and another extreme point $q_\alpha(p)$ of $\mathcal P$. We call $q_\alpha(p)$ an \emph{adjacent} extreme point of $p$. See \cref{fig:edge-inspiration} for an illustration of the idea of the more familiar setting of finite-dimensional space. Let's make all of this more precise.

\begin{figure}
    \centering
    \includegraphics[width=0.4\linewidth]{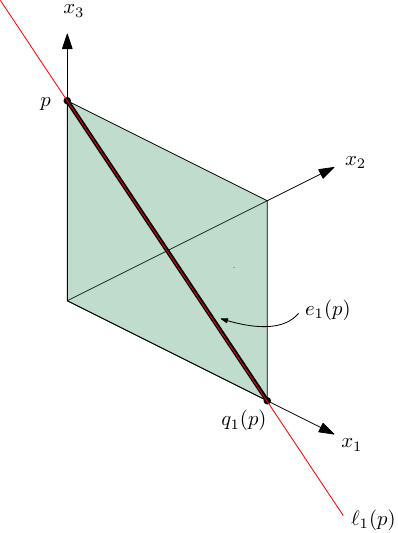}
    \caption{An illustration of extreme points, edges, and adjacent extreme points. Here $S_1 = \{ x \in \R^3 \mid -x_1 \le 0$, $S_2 = \{x \in \R^3 \mid -x_2 \le 0\}$, $S_3 = \{x \in \R^3 \mid -x_3 \le 0\}$, and $S_4 = \{ x \in \R^3 \mid x_1 + x_2 + x_3 \le 1\}$.}
    \label{fig:edge-inspiration}
\end{figure}

We call the convex cone $\mathcal S(p) \triangleq \cap_{\alpha\in A(p)} S_\alpha$ the \emph{support cone} at extreme point $p \in E$ and $\mathcal R(p) \triangleq \{d\in X \mid \phi_\alpha (p+\lambda d)\leq b_\alpha \mbox{ for all } \lambda \ge 0 \mbox{ and } \alpha\in A(p)\}$ the \emph{recession cone} at extreme point $p \in E$. Note that $\mathcal S(p)$ has its apex at $p$, whereas $\mathcal R(p)$ has its apex at the origin $0$.

\begin{lemma}\label{lemma:redef-recession-cone}
The recession cone $\mathcal R(p)= \{ d\in X  \mid  \phi_\alpha (d)\leq 0 \mbox{ for all }\alpha \in A(p)\}$.    
\end{lemma}

\cref{lemma:redef-recession-cone} tells us that $\mathcal R(p)= \mathcal S(p) - p$; that is, $\mathcal R(p)$ is $\mathcal S(p)$ translated to the origin. 

Note that every $d\in \mathcal R(p)$ with $d\ne 0$ generates a half-line or {\em ray} $r$ of $\mathcal R(p)$, i.e.  $r=\{\lambda d, \lambda \geq 0\}$.
A ray $r$ is is an {\em extreme ray} of a convex cone $C$  with vertex 0  if $x\in r, y\in C$ and $x-y\in C$ implies $y\in r$ \cite[p. 72]{schaefer1971topological}.
A half-line $r\in X$ is a ray or extreme ray of a convex cone $C$ with vertex $p$ if $r-p$ is a ray or, respectively extreme ray, of $C-p$.
Let $\eta_{\alpha}(p)$ for $\alpha\in A(p)$ be the ray in $\mathcal R(p)$ given by $\eta_{\alpha}\triangleq \ell_\alpha(p)\cap S_\alpha-p.$ Note that $\eta_\alpha(p)\ne \{0\}$.

\begin{lemma}\label{lemma:extreme-rays}
If \cref{ass:compact,ass:support-conditions,ass:constraints-bounded} hold then rays $\eta_{\alpha}(p),  \alpha\in A(p)$ are the extreme rays of $\mathcal R(p)$.
\end{lemma}

Since $\mathcal R(p)$ is $\mathcal S(p)-p$,  we have the following theorem from \cref{lemma:extreme-rays}.

\begin{lemma}
If \cref{ass:compact,ass:support-conditions,ass:constraints-bounded} hold then rays $ r_\alpha (p)= \eta_\alpha(p) + p, \alpha \in A(p)$ are the extreme rays of $\mathcal S(p)$.
\end{lemma}

\begin{theorem}[Adjacent extreme points, $q_\alpha(p)$]\label{theorem:adjacent-extreme-points} Suppose \cref{ass:compact,ass:support-conditions,ass:constraints-bounded} hold. For each $\alpha \in A(p)$, let $q_\alpha(p) \triangleq p + \lambda^* t$ where $\lambda^* =  \min_{\lambda \ge 0} \{\lambda | p + \lambda r_\alpha(p) \in \cap_{\beta \in A'(p)} S_\beta\}$ where $A'(p) \triangleq \{ \beta \in A \setminus A(p) \mid r_\alpha(p) \cap H_\beta \neq \emptyset\}$ and $t$ is some point in the half-line $r_\alpha(p)$. Then $q_\alpha(p)$ is an extreme point of $\mathcal P$.  
\end{theorem}
\begin{proof}
The set $A'(p)$ indexes the set of hyperplanes indexed by $A \setminus A(p)$ that intersect with $r_\alpha(p)$. These intersections are points by \cref{lemma:0}. Observe that $q_\alpha(p)$ is the closest point to $p$ of intersections of hyperplanes in $A'(p)$ with $r_\alpha(p)$. This closest point $q_\alpha(p)$ is the intersection of $r_\alpha(p)$ with the boundary of $\cap_{\alpha\in A'(p)} S_\alpha$---which being the boundary of the intersection of a collection of closed halfspaces (since $\phi_\alpha, \alpha \in A'(p) $ are continuous linear functions)---is closed. Note that there must be at least one point of intersection since $\mathcal P$ is linearly bounded. This closest point must be distinct from $p$ by \cref{ass:support-conditions,ass:constraints-bounded} and the proof of \cref{prop:extreme-then-intersect}. Note that $q_\alpha (p)$ is an extreme point of $\mathcal P$ from \cref{thm:characterize-EP} since it is the unique point contained in the constraining hyperplanes passing through it.
\end{proof}

The extreme points $\{q_\alpha(p) \mid \alpha \in A(p)\}$ are called the {\em adjacent} extreme points of $p$. The line segment $e_\alpha(p)=q_\alpha(p)-p$ joining $q_\alpha(p)$ and $p$ is called an {\em edge direction} of $\mathcal P$ for $\alpha\in A(p)$.  The set $\mathcal E(p)$ of all edge directions out of $p$ is given by  $\mathcal E(p)= \cup_{\alpha\in A(p)} e_\alpha (p)$ and the set $\mathcal E$ of all edges in $\mathcal P$  is given by 
$\mathcal E= \cup_{p\in E}\mathcal E(p)$.

\subsection{Edges from an extreme point form a Schauder basis}

We now establish some key properties on the set of edges emanating from an extreme point that are essential in defining our simplex method. In particular, we will describe how these edge directions can be used to fashion a Schauder basis for the space and the Schauder basis coefficients are nonnegative when generating elements in the feasible region $\mathcal P$. This machinery is essential for showing the convergence of the simplex method we define in the next section. First, an assumption. 

\begin{assumption}[Countably many constraints]\label{ass:countably-many-constraints} 
The set $A$ is countable. In this case, we let $k$ be the index of $A$ and express $\mathcal P=\cap_{k=1}^\infty S_k$. In other words, from now, we use the index $k$ instead of the index $\alpha$.
\end{assumption}

\begin{remark}
One sufficient condition for \cref{ass:countably-many-constraints} to hold is that $X$ is a separable Banach space \cite[Proposition 7.6.5]{borwein2010convex}). In this case, even if $\mathcal P$ was expressed with more than countably many constraints, there exists a countable set of constraints that also describes $\mathcal P$.  
\end{remark}

\begin{remark}\label{rem:labeling-workaround}
We slightly abuse notation in our indexing of the constraints $\phi_k(x)$. We will re-index the set of active constraints at $p$ by $\phi_k(\cdot; p)$. In other words, we relabel the elements of $A(p)$ by $k = 1, 2, \dots$ and $\phi_k(\cdot; p)$ should be understood as the $k$ active constraint at $p$. 
\end{remark}

\begin{example}[Hilbert cube, continued]\label{ex:hilbert-cube-countably-many-constraints}
The index set $A = \{1,2,\dots\} \cup \{-1, -2, \dots\}$ of the Hilbert cube is countable, and so satisfies \cref{ass:countably-many-constraints}.
\end{example}

For the following development, refer to \cref{fig:Schauder-basis}.%
\begin{figure}
    \centering
    \includegraphics[width=0.5\linewidth]{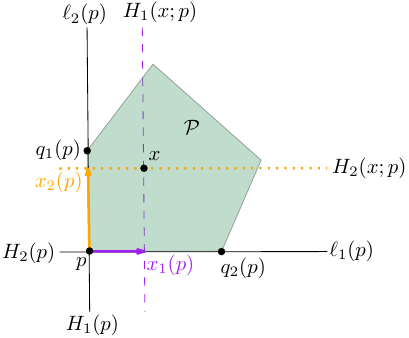}
    \caption{Visualizing the construction of the Schauder basis at an extreme point.}
    \label{fig:Schauder-basis}
\end{figure}
Let $p$ be an extreme point of a polytope $\mathcal P$. We begin by assuming $p=0$. Fix $x\in \mathcal P$ and let $x_k(0)$ be the vector along the $k$th edge out of $0$ given by
\begin{equation}\label{eq:define-x-sub-k}
x_k(0) \triangleq \ell_k(0)\cap H_k(x;0) 
\end{equation}
where $\ell_k(0)$ is the line defined in \cref{eq:emenating-line} and $H_k(x;0)=\{y\in X | \phi_k(y;0)=\phi_k(x;0)\}$ and $\phi_k(\cdot; 0)$.

We claim that
\begin{equation}\label{eq:little-to-show}
x_k(0)=\frac {\phi_k(x;0)} {\phi_k(q_k(0);0)} q_k(0)
\end{equation}
by showing that the right-hand side of \cref{eq:little-to-show} is in $\ell_k(0) \cap H_k(x;0)$ and is thus equal to $x_k(0)$ by the definition in \cref{eq:define-x-sub-k}.  The right-hand side of \cref{eq:little-to-show} is clearly in $\ell_k(0)$. Moreover, note that  
\[\phi_k\left(\frac {\phi_k(x;0)} {\phi_k(q_k(0);0)} q_k(0);0\right) =(\phi_k(x;0)/\phi_k(x_k(0);0))\phi_k(x_k(0);0)=\phi_k(x;0)\] 
and so the right-hand side of \cref{eq:little-to-show} is in $H_k(x;0)$. This establishes \cref{eq:little-to-show}. 

Define $x^n(0)\triangleq \sum_{k=1}^n x_k(0)$. 

\begin{lemma}\label{lemma:x-sub-k-get-more-feasible}
If \cref{ass:compact,ass:support-conditions,ass:countably-many-constraints,ass:constraints-bounded} hold then the sequence $x_n(0) \in \cap_{k=1}^n H_k(x;0)$ for all $n = 1,2,\dots$.
\end{lemma}

\begin{lemma}\label{lemma:technical-convergence-lemma}
Let $y\in X$. If $\cap_{k=1}^\infty H_k(y;0) = \{y\}$ and $\{y^n\}$ is a convergent sequence with $y^n\in \cap_{k=1}^n H_k(y;0)$ for $n=1,2,\ldots$,  then $y^n \rightarrow y$ as $n\rightarrow \infty$.
\end{lemma}

To ensure our Schauder basis construction converges, we need an additional assumption.

\begin{assumption}[Compactness of the set of partial sums]\label{ass:partial-sums-compact}
For all $p \in \mathcal P$ and $x \in X$, the set $\{x^n(0)\}$ lies in a compact set (which may depend on $x$ and $p$, but not on $n$).
\end{assumption}

\begin{remark}\label{rem:convergence-enough}
A sufficient condition for \cref{ass:partial-sums-compact} to hold is that the sequence $x^n(p), n=1,2,\ldots$ be convergent since then we can set the containing compact set equal to the closure of $ \{x^1(p), x^2(p), \ldots\}.$ This is a necessary as well as sufficient condition for \cref{ass:partial-sums-compact} since $x^n(p), n=1,2,\ldots$ is convergent under \cref{ass:partial-sums-compact} as we show in \cref{prop:schauder-basis-theta} below. Also, since $\phi_k(x;p)$ and $\phi_k(q_k(p);p)$ are respectively  bounded from above and below, a sufficient condition is that  $\sum_{k=1}^\infty \|q_k(p)\|$ is absolutely convergent and $X$ is some Banach space with norm $\|\cdot\|$. $\blacktriangleleft$
\end{remark}

\begin{remark}
In the case where $X$ is a reflexive Banach space (such as a Hilbert space or $L_p$ space with $1 < p < \infty$), the condition of compactness can be relaxed to uniform boundedness. This yields the (weakly) convergent subsequence needed in the proof of \cref{prop:schauder-basis-theta} below, which carries through to the rest of the paper. 
\end{remark}

The following example exhibits another approach to establishing \cref{ass:partial-sums-compact}.

\begin{example}[Hilbert cube, continued]\label{ex:hilbert-cube-partial-sums-compact}
We claim that $\mathscr H$ satisfies \cref{ass:partial-sums-compact}. Fix the extreme point $p=0$ (the argument is analogous for other extreme points) and note that $q_k(0)=e_k$ (the $k$th unit vector). Accordingly, for any $x=(x_j)\in \mathscr H$, $x_k(0) = x_k e_k$ and so the associated $k$th partial sum
\[
x^n(0) = (x_1,\dots,x_n,0,0,\ldots).
\]
Each $x^n(0)\in \mathscr H$ since all coordinates lie in $[0,1]$. Moreover,
\[
\|x-x^n(0)\|_X^2
=\sum_{k=n+1}^\infty \delta_k^2 x_k^2 \;\longrightarrow\; 0
\quad\text{as }n\to\infty,
\]
because $x\in X$ implies $\sum_{k=1}^\infty \delta_k^2 x_k^2<\infty$. Therefore $x^n(0)\to x$ , so the set $\{x_n(0):n\ge 1\}$ is closed and thus compact since it lies in the compact set $\mathscr H$. This establishes \cref{ass:partial-sums-compact}. $\blacktriangleleft$
\end{example}

\begin{proposition}[Schauder basis formed at origin]\label{prop:schauder-basis-theta} Suppose
\cref{ass:compact,ass:support-conditions,ass:constraints-bounded,ass:countably-many-constraints,ass:partial-sums-compact} hold. Let $x\in \mathcal P$ and extreme point $p=0$. The sequence $x^n(0) \rightarrow x$ as $n\rightarrow \infty$, that is, 
$x=\sum_{k=1}^\infty\frac {\phi_k(x;0)} {\phi_k(q_k(0);0)}\, q_k(0)$
and $\{q_k(0), k=1,2,\ldots\}$ is a Schauder basis for the elements of $\mathcal P$.
\end{proposition}
\begin{proof}
Recall $x^n(0)=\sum_{k=1}^n x_k(0)$.  %Suppose $x^n(0) \nrightarrow x$ as $n \rightarrow \infty$. Then, there is a subsequence $\{x^{n_\ell}(0)\}\notin U$ for some open set $U$ containing $x$. By compactness in \cref{ass:partial-sums-compact}, there is a subsequence $\{x^{n_{\ell_j}}(0)\}$ of $\{x^{n_\ell}(0)\}$ with $x^{n_{\ell_j}}(0)\rightarrow \bar{x}$   where $\bar{x} \notin U$ since the complement of $U$ is closed and hence $\bar{x}\ne x$. 
%Set $\bar{x}^i=x^{n_{i_j}+1}(0)$ 
%for $n_{i_j}+1 \le i \le n_{\ell_{j+1}}
%$ for $i=1,2,\ldots $. 
%Note that ${\bar{x}}^i \rightarrow \bar x$ as $i\rightarrow \infty$ and ${\bar{x}}^i \in \cap_{k=1}^i H_k(x;0) $ for all $i$. Since $\cap_{k=1}^\infty H_k(\bar{x};0)=\{\bar{x}\}$, we have by \cref{lemma:technical-convergence-lemma} that ${\bar{x}}^i\rightarrow x$ as $i\rightarrow \infty$. Contradiction since $x\ne \bar{x}$. Hence $x^n(0) \rightarrow x$ as $n\rightarrow \infty$. 
%
Suppose the sequence $(x^n(0))$ does not converge to $x$. Then there exists an open neighborhood $U$ of $x$ such that there is a subsequence $(x^{n_\ell}(0))$ entirely outside $U$. By the compactness condition in \cref{ass:compact}, this sequence must have an accumulation point $\overline{x}$. Because the complement of $U$ is closed, $\overline{x} \notin U$, and therefore $\overline{x} \ne x$. However, because $x^{n_\ell}(0) \in \bigcap_{k=1}^{n_l} H_k(x;0)$, and the hyperplanes are closed, it follows from \cref{lemma:technical-convergence-lemma} that any accumulation point of these partial sums must be $x$. This contradicts $\overline{x} \ne x$. Hence, $x^n(0) \to x$ as $n \to \infty$.
\end{proof}

\begin{theorem}[Schauder basis formed at a given extreme point]\label{theorem:schauder-basis}
Suppose
\cref{ass:compact,ass:support-conditions,ass:constraints-bounded,ass:countably-many-constraints,ass:partial-sums-compact} hold. Let $x\in \mathcal P$ and $p \in E$. Then  
\[x=\sum_{k=1}^\infty \frac {\phi_k(x-p;p)} {\phi_k(q_k(p)-p;p)} (q_k(p)-p)+p\]
where the coefficients $\frac {\phi_k(x-p;p)} {\phi_k(q_k(p)-p;p)}$ on the basis vectors $q_k(p) - p$ are nonnegative.
\end{theorem}
\begin{proof}
Since topological vector spaces are invariant under translations, we can replace $x$ by $x-p$ and substitute in \cref{prop:schauder-basis-theta} above. 
\end{proof}

\section{A geometric simplex algorithm}\label{s:geometric-simplex-method}

Consider the following linear program

\begin{align}\label{eq:linear-program}
\begin{split}
\min \ \  & c(x) \\
\text{s.t.  } & a_k(x) \le b_k, \text{ for } k = 1,2,\dots
\end{split}
\end{align}
where $c:X\rightarrow \R$ is a continuous linear functional. We assume that the feasible region $\mathcal P = \{x : a_k(x) \le b_k, \text{ for } k = 1,2,\dots\}$ satisfies our assumptions up until this point in the paper. Note that an optimal solution $x^*$ is attained since $c$ is continuous and $\mathcal P$ is compact by \cref{ass:compact}. %Suppose without loss of generality that $c(x^*) \ge 0$ (if not, set $c'(x) = -c(x), x\in \mathcal P$  and consider $\max c' (x), x \in \mathcal P$). \chris{Not sure where we need this last part.}

What follows is a simplex algorithm for solving for $x^*$.  It successively moves from an extreme point along a steepest descent edge to its adjacent extreme point. Let's formalize the notion of steepest descent edge. Consider the lines $\ell_k(p), k=1,2,\ldots$ with $p \in E$ as one-dimensional vector spaces of $X$. 
 Let $\| \cdot \|_{\ell_k(p)}$ be a norm for each vector space $\ell_k(p)$. Note that the topology on each space $\ell_k(p)$ is homeomorphic to that of the real line and not necessarily in agreement with the relative topology of $X$ when restricted to $\ell_k(p)$.

 We have that $(c(q_k(p))-c(p))/ \| q_k(p)-p \|_{\ell_k(p)}$ is the rate of cost decrease (or increase)  when moving along edge $k$ out of $p$. Note that $\|q_k(p)-p\|_{\ell_k(p)}>0$ since $q_k(p)\ne p$ for all $p\in E$. The {\em steepest descent edge} out of $p$ would be an edge $k$ that minimizes  $(c(q_k(p)) -c(p))/\| q_k(p)-p\|_{\ell_k(p)}$. We next assume that this minimum exists.

\begin{assumption}[Steepest descent edge exists]\label{ass:steepest-exists}
For each extreme point $p$ the following infimum is attained:
\begin{align*}
\gamma(p) := \min_{k=1,2,\dots} \frac{c(q_k(p))-c(p)}{\| q_k(p)-p\|_{\ell_k(p)}}.
\end{align*}
We call a resulting edge a \emph{steepest descent edge} emanating from $p$. 
\end{assumption}

\begin{remark}\label{rem:closed-suffices}
A sufficient condition for \cref{ass:steepest-exists} is that for every extreme point $p$, the set of adjacent extreme points to $p$ (and including $p$) are a closed set, since then by Weierstrass extreme value theorem, the minimization problem defining $\gamma(p)$ has an optimal solution.
\end{remark}

\begin{example}[Hilbert cube, continued]\label{ex:hilbert-cube-steepest-exists}
The extreme points of $\mathscr H$ are the $\{0,1\}$–valued
sequences (each belongs to $X$ because $\sum_j \delta_j^2<\infty$). Fix an extreme point $p\in \mathcal P$ and, for each $k\in\mathbb{N}$, let $q_k(p)$ be the adjacent extreme point obtained by
flipping the $k$th coordinate of $p$ from $0$ to $1$ or $1$ to $0$. Then $\|q_k(p)-p\|_X=\delta_k\ \rightarrow\ 0$ as $k \to \infty$. Hence $q_k(p)\to p$ in $\mathscr H$. This implies the set $E(p)$ is closed for all $p \in E$. Thus, via \cref{rem:closed-suffices}, \cref{ass:steepest-exists} is satisfied. $\blacktriangleleft$
\end{example}

It is important for the geometric simplex algorithm to select edges of steepest descent in order to reach an optimal solution (or converge in value to optimality) as we show in \cref{theorem:convergence-in-value}. The issue is illustrated in the binary tree example of \cite[p. 869-7]{ghate2010shadow}. There, the authors illustrate that if a simplex method takes a non-steepest descent choice initially, it can thereafter make steepest descent choices and not converge in value to optimality. This is because it can get ``stuck'' on an infinitely improving path that converges in value to a suboptimal level. 

We are now ready to state the ``geometric'' simplex algorithm, the core mechanism studied in this paper:

\bigskip{\bf Simplex Algorithm}
\begin{enumerate}
\item Choose an initial extreme point $p\in \mathcal P$ and set $n=0$ and $p^0=p$.
\item {\bf while} $\gamma(p^n)<0$ {\bf do} 
\begin{quote}
Find  the adjacent extreme point  $p^{n+1}$ along a steepest descent edge from $p^n$, that is,
 \[\mbox{Set } p^{n+1}=\arg \min_{k=1,2,\ldots} \frac{c(q_k(p^n)) - c(p^n)}{\| q_k(p^n)-p^n\|_{\ell_k(p^n)}}. \] 
 Set $n=n+1$.
\end{quote}
\item  {\bf end while}
\end{enumerate}
\begin{remark}
Note that the Simplex Algorithm moves or pivots from $p^n$ to $p^{n+1}$ by removing exactly one of the hyperplanes passing through $p^n$ and adding exactly one of the hyperplanes not passing through $p^n$ during each iteration $n$.
\end{remark}

We turn now to establishing that there is a strictly positive minimum cost reduction in each iteration of the simplex algorithm. This requires the following assumption.

\begin{assumption}[Edge lengths bounded away from zero and infinity]\label{ass:edge-lengths-bounded-away-from-zero}
$\nu = \inf_{p \in E, k \in A(p)} \|e_k(p)\|_{\ell_k(p)} > 0 $ and $D \triangleq\sup_{p \in E, k \in A(p)} \|e_k(p)\|_{\ell_k(p)} < \infty$.
\end{assumption}

\cref{ass:edge-lengths-bounded-away-from-zero} ensures that each edge direction 
has a minimal geometric “thickness,” preventing degenerate cases where the distance between adjacent extreme points tends to zero. Without this, the simplex path could accumulate infinitely many pivot steps in a bounded region without improving the objective sufficiently to guarantee convergence in value.

\begin{example}[Hilbert cube, continued]\label{ex:hilbert-cube-edge-lengths-bounded-away-from-zero}
Following \cref{ex:hilbert-cube-steepest-exists}, we know every extreme point $p$ has the form $p_i = 0$ or $1$ for all $i$, and the set of adjacent extreme points to $p$ has elements $q_k(p)$ where the $k$th coordinate of $p$ is ``flipped'' from $0$ to $1$ or $1$ to $0$. The edge $e_k(p)$ is thus the unit vector with a $1$ is the $k$th position and $0$ otherwise, and $\ell_k(p)$ is the $k$th ``coordinate'' axis. Thus, $||\cdot||_{\ell_k(p)}$ is the usual topology on the real line. Thus, $\|e_k(p)\|_{\ell_k(p)} = 1$ for all $k$ and all $p$ and so $\nu = 1$.  $\blacktriangleleft$
\end{example}

\begin{lemma}\label{lemma:termination-condition}
If \cref{ass:compact,ass:support-conditions,ass:constraints-bounded,ass:countably-many-constraints,ass:partial-sums-compact,ass:steepest-exists,ass:edge-lengths-bounded-away-from-zero} hold then either the Simplex Algorithm terminates at some iteration $n$ or $\lim_{n\rightarrow\infty}\gamma(p^n)=0$.    
\end{lemma}

For the remainder of the argument, we need another uniformity property to complement \cref{ass:constraints-bounded}.

\begin{assumption}[Constraint functions uniformly bounded over extreme points] $\xi \triangleq \sup_{p \in E, \alpha \in A} |\phi_\alpha(p)| < \infty$.\label{ass:bounded-at-extremes}
\end{assumption}

Note that by Bauer's Maximum Principle \cite[Theorem~7.69]{hitchhiker}, this implies that the supremum over the whole set $\mathcal P$ is also bounded by $\xi$, since an optimal solution exists among the extreme points of $\mathcal P$ when optimizing over any continuous linear functional. 
\cref{ass:bounded-at-extremes} also prevents pathological scaling of the constraint system that could destroy the uniformity of convergence arguments later in \cref{theorem:convergence-in-value}.

\begin{example}[Hilbert cube, continued]\label{ex:hilbert-cube-bounded-at-extremes}
Recall from \cref{ex:hilbert-cube} that $\phi_k(x) = \langle x, y^k \rangle$ where $y^k = \pm 1$ for all $k\in A$. As we have also observed, every extreme $p$ has the form $p_j = 0$ or $1$ for all $j$. Thus, we can see that $\phi_k(p) \le 1$ for all $p \in E$. In particular $\xi = 1 < \infty$ and so \cref{ass:bounded-at-extremes} is satisfied. $\blacktriangleleft$
\end{example}

In order to assure convergence in value of the Simplex Algorithm, we also need an assumption on the objective function:

\begin{assumption}[Uniform convergence of edge costs]\label{ass:uniform-convergence-edge-costs}
Let $c(q_k(p) -p)$ be the cost of edge $e_k(p)$ for $p\in E$ and $k=1,2,\dots$. We assume $\sum_{k=1}^\infty |c(q_k(p) -p)|$ converges uniformly over $p\in E$.
\end{assumption}

\begin{example}[Hilbert cube, continued]\label{ex:hilbert-cube-ass:uniform-convergence-edge-costs} Let $c:X\to\mathbb{R}$ be a continuous linear functional, represented by some $h=(h_j)\in X$
via the Riesz map: $c(x)=\langle h,x\rangle_X=\sum_{j\ge1}\delta_j^2 h_j x_j$.
For an extreme point $p$ of $H$ and its adjacent extreme point $q_k(p)$ (obtained by flipping the $k$th
coordinate), the edge is $q_k(p)-p=\pm e_k$, so $c(q_k(p) -p)=\langle h,e_k\rangle_X= \pm\delta_k^2 h_k$ independent of $p$ up to sign. Therefore $\sum_{k=1}^\infty |c(q_k(p) -p)|=\sum_{k=1}^\infty \delta_k^2\,|h_k|$
and this series converges \emph{uniformly in $p\in E$} (indeed, the right-hand side does not depend on $p$) whenever $\sum_{k\ge1}\delta_k^2 |h_k|<\infty$. Hence, under the mild additional summability condition $\sum_{k}\delta_k^2 |h_k|<\infty$ on the Riesz representative $h$ of $c$, \cref{ass:uniform-convergence-edge-costs} holds.
\end{example}

In finite dimensions, the simplex method can optimize \emph{any} linear functional over a polyhedron. For a simplex method to work in infinite spaces, we need both properties on $\mathcal P$ and $c$, since a linear functional is not nearly as well-behaved in such spaces as in finite-dimensional space. \cref{ass:uniform-convergence-edge-costs} aims to offer a minimal restriction on $c$ with respect to the set of extreme points $E$.

\begin{remark}\label{rem:strongest-cost-convergence}
Note that $| c(q_k(p) - p) | \le C \alpha^k$ for all $p\in E$ and $k=1,2,\ldots$ for some $0< \alpha<1$ and  constant $C>0$ is a sufficient condition for \cref{ass:uniform-convergence-edge-costs} to hold.
\end{remark}
 
\begin{theorem}[Convergence in value]\label{theorem:convergence-in-value}
If \cref{ass:compact,ass:support-conditions,ass:constraints-bounded,ass:countably-many-constraints,ass:partial-sums-compact,ass:steepest-exists,ass:edge-lengths-bounded-away-from-zero,ass:uniform-convergence-edge-costs,ass:bounded-at-extremes} hold then either the Simplex Algorithm terminates with an optimal solution or the infinite sequence of extreme points generated converges in value to an optimal solution; i.e. 
$\lim_{n \rightarrow \infty} c(p^n) = c^*$,
where $c^*=c(x^*)$ and $x^*$ is an optimal solution of \cref{eq:linear-program}.
\end{theorem}

\begin{proof} Suppose the Simplex Algorithm terminates at iteration $n$. Then $ \gamma(p^n)=\min_{k=1,2,\dots} (c(q_k(p^n))-c(p^n)) /\| q_k(p^n)-p^n\|\ge 0$
so that  
$c(q_k(p^n))-c(p^n)\ge 0$ for all $k=1,2,\ldots.$
From \cref{theorem:schauder-basis}, for any $x\in \mathcal P, x\ne p^n$, we have
\[x=\sum_{k=1}^\infty\frac {\phi_k(x-p^n;p^n)} {\phi_k(q_k(p^n)-p^n;p^n)} (q_k(p^n)-p^n)+p^n\]
and since $c$ is a continuous function,
\begin{equation}\label{eq:cost-breakdown}
c(x)-c(p^n)=\sum_{k=1}^\infty\frac {\phi_k(x-p^n;p^n)} {\phi_k(q_k(p^n)-p^n;p^n)} (c(q_k(p^n))-c(p^n)).
\end{equation}
But $\phi_k(x-p^n;p^n) \le 0$ and $\phi_k(q_k(p^n)-p^n;p^n) < 0$ since $x$ and $q_k(p^n)$  are in $\mathcal P$ and $q_k(p^n)$ is not equal to $p^n$.
Hence $c(x)-c(p^n)\ge 0$ so that $c(p^n)\le c(x)$ and therefore $p^n$ is optimal.

\medskip
Now, suppose the Simplex Algorithm does not terminate.  Consider $x^*\in \mathcal P$. We have
\begin{equation*}
c(x^*)-c(p^n)=\sum_{k=1}^\infty\frac {\phi_k(x^*-p^n;p^n)} {\phi_k(q_k(p^n)-p^n;p^n)} (c(q_k(p^n)-c(p^n)).    
\end{equation*}
Let $\epsilon>0$. We have by \cref{ass:bounded-at-extremes}
\[
|\phi_k(x^*-p^n;p^n)| \le  2\xi < \infty
\]
using the triangle equality. 

Moreover, $\phi_k(q_k(p^n) - p^n;p^n) = \phi_k(q_k(p^n);p^n) - \phi_k(p^n;p^n) = 
\phi_k(q_k(p^n);p^n) - b_k(p^n) = -s_k(q_k(p^n);p^n).$
Thus  $s_k(q_k(p^n)) = -\phi_k(q_k(p^n) - p^n) >0.$ But $s_k(q_k(p^n)) \ge \inf_{p\in E, \alpha\in A\setminus A(p)} s_k(p) = \rho > 0$ by \cref{ass:support-conditions}. Hence $|\phi_k(q_k(p^n)-p^n;p^n)\ge \rho >0$. By \cref{ass:uniform-convergence-edge-costs}, for all $\epsilon' >0$ there is a $K' (\epsilon') $ such that $\sum_{k=1}^\infty |c(e_k(p))| < \epsilon$ for all $p\in E$.

Then 
\begin{align*}
&\sum_{k = K' (\epsilon')}^\infty \frac {\phi_k(x^*-p^n;p^n)} {\phi_k(q_k(p^n)-p^n;p^n)} (c(q_k(p^n)) - c(p^n)) \\
\le & \sum_{k = K' (\epsilon')}^\infty \frac {|\phi_k(x^*-p^n;p^n)|} {|\phi_k(q_k(p^n)-p^n;p^n)|} |c(q_k(p^n)) - c(p^n))| \\
\le & \frac {\xi D} {\rho}  \sum_{k = K' (\epsilon')}^\infty |c(q_k(p^n)) - c(p^n))| \\
= & \frac {2\xi D} {\rho} \sum_{k = K' (\epsilon')}^\infty |c(e_k(p^n))| \\
< & \frac {2\xi D} {\rho} \epsilon'.
\end{align*}
Set $\epsilon' = \rho
\epsilon/2\xi D$,  $K(\epsilon) = K' (\epsilon')$ and we have
\[
\sum_{k = K (\epsilon)}^\infty \frac {\phi_k(x^*-p^n;p^n)} {\phi_k(q_k(p^n)-p^n;p^n)} c(q_k(p^n) - c(p^n)) < \epsilon
\]
for all $\epsilon >  0$ and $p^n, n=1,2,\ldots$. That is, 
\[
\sum_{k = 1}^\infty \frac {\phi_k(x^*-p^n;p^n)} {\phi_k(q_k(p^n)-p^n;p^n)} (c(q_k(p^n))- c(p^n))
\]
converges uniformly over $p^n, n=1,2,\ldots.$

Then for any $\epsilon>0$, there is a $K(\epsilon)$ independent of $p^n$ such that 
\[ -\epsilon < c(x^*)-c(p^n) - \sum_{k=1}^{K(\epsilon)} \frac {\phi_k(x^*-p^n;p^n)} {\phi_k(q_k(p^n)-p^n;p^n)} (c(q_k(p^n)-c(p^n)) <\epsilon\]
so
\begin{equation}\label{eq:first}
 c(x^*)-c(p^n) > \sum_{k=1}^{K(\epsilon)} \frac {\phi_k(x^*-p^n;p^n)} {\phi_k(q_k(p^n)-p^n;p^n)} (c(q_k(p^n)-c(p^n)) - \epsilon
\end{equation}
for all $p^n, n=1,2,\ldots$.

We have by \cref{lemma:termination-condition}, $\lim_{n\rightarrow\infty} \gamma(p^n) = 0$ while $\gamma(p^n) < 0$ for all $n$. Hence  for all $\epsilon' > 0$, there is a $N(\epsilon') > 0$  such that $0 >\gamma (p^n) > -\epsilon'$ for all $n> N(\epsilon')$
Hence for 
$n>N(\epsilon')$, 
\[0> \min_{k=1,2,\ldots}(c(q_k(p^n))-c(p^n))/ \|q_k(p^n)-p^n \|_{\ell_k(p^n)} >- \epsilon'\]
so that 
\[0> (c(q_k(p^n))-c(p^n))/ \|q_k(p^n)-p^n \|_{\ell_k(p^n)} > - \epsilon'\]
or
\begin{equation}\label{eq:second}
c(q_k(p^n))-c(p^n)>  - \|q_k(p^n)-p^n \|_{\ell_k(p^n)}  \epsilon' > - D\epsilon'
\end{equation}
for all $k=1,2,\ldots$ and $n > N(\epsilon')$.

From \cref{eq:first,eq:second}, 
\begin{equation}\label{eq:third}
c(x^*)-c(p^n) > \sum_{k=1}^{K(\epsilon)} \frac {\phi_k(x^*-p^n;p^n)} {\phi_k(q_k(p^n)-p^n;p^n)} (- D \epsilon') - \epsilon
\end{equation}
for all $n > N(\epsilon')$
since $\phi_k(x^*-p^n;p^n) \le 0$ and $\phi_k(q_k(p^n)-p^n;p^n) < 0$ because $x^*\in \mathcal P$ and $q_n(p^n)\in \mathcal P$ and $q_n(p^n)\ne p^n$.

Hence from \cref{eq:third}, 
\[c(x^*)-c(p^n) > -D \epsilon' \sum_{k=1}^{K(\epsilon)} \frac {\phi_k(x^*-p^n;p^n)} {\phi_k(q_k(p^n)-p^n;p^n)}   - \epsilon\]
for all $n > N(\epsilon')$.

Let ${\bar \phi}(\epsilon,n,x^*) = \max_{k=1,2,\ldots, K(\epsilon)} \phi_k(x^*-p^n;p^n)$ and ${\underline \phi}(\epsilon,n,x^*) = \min_{k=1,2,\ldots, K(\epsilon)}\phi_k(q_k(p^n)-p^n;p^n)$.

Then
\[ c(x^*)-c(p^n) > -D \epsilon' K(\epsilon) \frac {{\bar \phi}(\epsilon,n,x^*) } {{\underline \phi}(\epsilon,n,x^*)}   - \epsilon \]
for all $n > N(\epsilon')$. Set $\epsilon' = (\epsilon {\bar \phi}(\epsilon,n,x^*)) / (2 D K(\epsilon) {\underline \phi}(\epsilon,n,x^*))$. Then $c(x^*) - c(p^n) > - \epsilon/2 -\epsilon/2 = - \epsilon$ for all $n > N(\epsilon')$. Since $\epsilon > 0$ was arbitrary, $\lim_{n\rightarrow\infty} c(p^n) = c(x^*)$.
\end{proof}

Note that since $\mathcal P$ is compact and the sequence $(p^n)_{ n=1,2,\ldots}$ is contained in $\mathcal P$, the sequence must have an accumulation point $p \in P$.
%is a subsequence $p^{n_\ell}, \ell=1,2,\ldots$ and a $p\in E$ with $p^{n_\ell} \rightarrow p$ as $\ell \rightarrow \infty$. 
By \cref{theorem:convergence-in-value} and the continuity of $c$, we have $c^*= \lim = c(p^{n_\delta})$ where $n_\delta$ is a converging subnet, so that $p$ is an optimal solution. 
Below we show that when the optimal solution is unique, we obtain a sequence of pivots $p^n$  that converges to $x^*$.

\begin{corollary}\label{cor:convergence-in-solution}
Let \cref{ass:compact,ass:support-conditions,ass:constraints-bounded,ass:countably-many-constraints,ass:partial-sums-compact,ass:steepest-exists,ass:edge-lengths-bounded-away-from-zero,ass:uniform-convergence-edge-costs,ass:bounded-at-extremes} hold and suppose the optimal solution $x^*$ is unique. Either the Simplex Algorithm terminates with the optimal solution $x^*$ or the infinite sequence of extreme points generated converges to the optimal solution $x^*$, i.e. $\lim_{n \rightarrow \infty}p^n = x^*$.
\end{corollary}
\begin{proof}
From \cref{theorem:convergence-in-value}, either the simplex algorithm terminates in an optimal solution or $\lim_{n \rightarrow \infty} c(p^n) = c^*$
where $c^*=c(x^*)$. Suppose $p^n \nrightarrow x^*$ as $n\rightarrow \infty$. Then there is an open neighborhood $U$ containing $x^*$  such that a subsequence $(p^{i_j})_{ j=1,2,\ldots}$ lies entirely outside of $U$. But $\mathcal P$ is compact, the closed subest $\mathcal P \setminus U$ is also compact. Therefore, the sequence $(p^{i_j})_{ j=1,2,\ldots}$ has an accumulation point $x' \in \mathcal P \setminus U$. By the continuity of $c$ and \cref{theorem:convergence-in-value}, evaluating $c$ along the subnet converging to $x'$ yields $c(x') = c^*$, making $x'$ an optimal solution. This contradicts the assumption that $x^*$ is uniquely optimal, since $x' \notin U$ ensures $x \neq x'$. 
%
%so there is a subsequence $p^{i_{j_k}}, j=1,2,\ldots$ 
%converging to a point $x' \in \mathcal P$, with $x' \ne x^*$. But $c(p^{i_{j_k}}\rightarrow c^*, j=1,2,\ldots$ and hence $x'$ is optimal. Contradiction since $x^*$ is uniquely optimal.
\end{proof}

\section{Defining polytopes}\label{s:defining-polytopes}

As mentioned in the introduction, there is a long history of attempting to define what we mean by a ``polytope'' in general topological space. As outlined, the usual notions from the finite-dimensional setting fail somewhat spectacularly. This gave rise to a long succession of attempts to define polytopes in infinite dimensions in a more workable manner \cite{bastiani1959cones,klee1956strict,maserick1965convex,alfsen1969vertices,phelps1969infinite,amir1972suns,gleit1972note,rajagopalan1974maximal,lau1974infinite,brosowski1974some,durier1986sets}. For example, Phelps \cite{phelps1969infinite} was interested in the Choquet ``simplex'' and attempted definitions of a polytope that captured this ``simplex''. However, a stunning set of results by \cite{armstrong1977polyhedrality} showed that nearly all of these attempted definitions excluded something quite simple: infinite-dimensional cubes. This is especially surprising from an optimization perspective, since a cube corresponds to ``box constraints'' in linear programming. One would hope that any definition of polytope captures such a foundational setting. We have shown in \cref{ex:hilbert-cube,ex:hilbert-cube-ass:uniform-convergence-edge-costs,ex:hilbert-cube-bounded-at-extremes,ex:hilbert-cube-partial-sums-compact,ex:hilbert-cube-steepest-exists,ex:hilbert-cube-edge-lengths-bounded-away-from-zero,ex:hilbert-cube-countably-many-constraints} that the Hilbert cube satisfies the assumptions of this paper's setting. We are thus the first paper that we are aware of to construct a simplex method that is amenable to the Hilbert cube. 

For concreteness, let us illustrate how a cube can present issues for an existing definition of a polytope in the literature. Consider the definition of Klee in \cite{klee1956strict}, which is arguably the easiest to state. Klee defines a polytope to be any closed convex set whose intersection with every finite-dimensional affine subspace (also called a ``flat'' in Klee's language) is a finite-dimensional polytope (in its usual definition) in that flat. The following example shows that the Hilbert cube $\mathscr H$ defined in \cref{ex:hilbert-cube} fails this definition in a spectacular fashion.

\begin{example}[Hilbert cube is not a Klee polytope]\label{ex:hilbert-cube-not-Klee}
We show that the Hilbert cube $\mathscr H$ defined in \cref{ex:hilbert-cube} is \emph{not} a polytope in Klee's sense by showing that a section of $H$ 
with a $2$-dimensional flat is not a finite-dimensional polytope in that flat. In fact, we can cut the Hilbert cube with a 2-dimensional flat and reveal a disc!

Let $p, u, v \in \R^\N$ where $u$ and $v$ are linearly independent (that is, there does 
not exist a $c \in \R$ such that $v = c u$). Then the set 
\begin{align*}%\label{eq:flat-in-the-space}
F_{uv}(p) = \left\{p + \alpha u + \beta v \mid \alpha, \beta \in \R\right\} 
\end{align*}
is a $2$-dimensional flat of $\R^\N$. The intersection $\mathscr H \cap F_{uv}(p)$ is an (at most $2$-dimensional) 
section of $\mathscr H$ given by
\begin{align*}%\label{eq:section-in-the-space}
\mathscr H \cap F_{uv}(p) = \left\{p + \alpha u + \beta v \mid 
0 \le p_i + \alpha u_i + \beta v_i \le 1, 
\text{ for all } \alpha, \beta \in \R \text{ and } i = 1,2,\dots \right\}. 
\end{align*}
Re-expressing the above set through its linearly isomorphic representation in the two-dimensional 
subspace $V$ of $\R^\N$ spanned by $\left\{u,v\right\}$ yields:
\begin{align}\label{eq:section-in-2d}
S_{uv}(p) = \left\{(\alpha,\beta) \in \R^2 \mid 
0 \le p_i + u_i\alpha  + v_i \beta \le 1, \text{ for all } i = 1,2,\dots \right\}. 
\end{align} 
The goal for the rest of the proof is to show that $u$, $v$ and $p$ can be chosen so 
that $S_{uv}(p)$ is the unit disk centered at the origin in $V$. Note that $S_{uv}(p)$ 
is a countable intersection of halfspaces in $\R^2$. 

The first insight is that the unit disk can also be expressed as the countable intersection 
of halfspaces. Let $C$ be the unit circle and $D$ be the unit disk in $\R^2$ and let 
\begin{align*}
E = \left\{(\alpha, \beta) \in \R^2 
: r \alpha + q \beta \le 1 \text{ for all } (r, q) \in C \cap \Q^2 \right\},
\end{align*} 
where $\mathcal Q$ is set of rational numbers. We claim that $D = E$. We first show that $D \subseteq E$. Let $T$ be any tangent line of 
$C$ in $\R^2$. By the convexity of $D$, it is clear that $D$ lies on entirely one side of 
$T$. Let $(r,q) \in C \cap \Q^2$ and note that the set 
$\left\{(\alpha, \beta) \mid r \alpha + q \beta = 1 \right\}$ is the tangent line to $C$ 
through $(r,q)$. It is clear that 
$D \subset \left\{(\alpha, \beta) \mid r \alpha + q \beta \le 1 \right\}$. Taking the intersection 
over all $(r,q) \in C \cap \Q^2$ yields $D \subseteq E$. 

Conversely, let $(a, b) \notin D$. Let $(\bar a, \bar b)$ be the point in $D$ with the smallest 
distance to $(a,b)$, where the distance derives from the Euclidean norm in $\R^2$. 
The tangent plane $\left\{(\alpha, \beta) \mid \bar a \alpha + \bar b \beta = 1 \right\}$ separates 
$D$ from $(a, b)$. In particular, we know $\bar a a+ \bar b b > 1$. Hence, there exist 
rationals $(\bar r,\bar q) \in C \cap Q^2$ sufficiently close to $(\bar a, \bar b)$ so that it 
remains that $\bar r a + \bar q b > 1$. Hence $(a,b) \notin E$. That is, $E \subseteq D$. 
This implies $E = D$.

Now, let $\bar u,\bar v$ be chosen such that $\bar u_i = r/2$ and $\bar v_i = q/2$ where $(r,q)$ 
span the countably many points $C \cap \Q^2$. Choose $\bar p$ such that $\bar p_i = 1/2$ for 
all $i \in \N$. We claim that $S_{\bar u\bar v}(\bar p) = E$. First we show 
$S_{\bar u\bar v}(\bar p) \subseteq E$. According to \cref{eq:section-in-2d} we have
\begin{align*}
S_{\bar u\bar v}(\bar p) = \left\{(\alpha,\beta) \in \R^2 
\mid -1 \le r \alpha + q \beta \le 1, \text{ for all } (r,q) \in C \cap \Q^2 \right\}
\end{align*}
after some rearranging. Hence it clearly follows that $S_{\bar u\bar v}(\bar p) \subseteq E$ 
since $S_{\bar u\bar v}(\bar p)$ includes all the inequalities defining $E$ 
(indeed all the $r \alpha + q \beta \le 1$ constraints). Conversely, we note that 
$E \subseteq S_{\bar u\bar v}(\bar p)$. This also follows immediately since if 
$r \alpha + q \beta \le 1$ is an inequality defining $E$ then so is $-r \alpha - q \beta \le 1$ 
since if $(r,q) \in C \cap \Q^2$ then $(-r, -q) \in C \cap \Q^2$. Rewriting the latter inequality 
yields $r \alpha  + q \beta \ge -1$. This argument for all $(r, q) \in C \cap Q^2$ recovers all 
of the constraints of $S_{\bar u\bar v}(\bar p)$. Hence, $E \subseteq S_{\bar u\bar v}(\bar p)$ 
and conclude $E = S_{\bar u\bar v}(\bar p)$.

Putting things together, our work shows $D = S_{\bar u\bar v}(\bar p)$. This implies that 
$\mathscr H \cap F_{\bar u\bar v}(\bar p)$ is linearly isomorphic to a unit disk. Clearly, the unit disk 
cannot be expressed as the intersection of finitely many halfspaces. Hence, $\mathscr H$ is not a Klee polyhedron.  $\blacktriangleleft$
\end{example}

\begin{remark}\label{rem:any-convex-set}
Note that if we replace the unit disk $D$ in the above example with \emph{any} finite-dimensional 
compact convex set, then a similar construction can be made and $p, u, v$ chosen so that $D$ is an 
affine section of $\mathscr H$. This follows from  \cite[Corollary 7.49]{hitchhiker} that every 
finite-dimensional compact convex set is the intersection of countably many closed 
halfspaces.   $\blacktriangleleft$
\end{remark}

We here put forward our definition of a polytope in infinite dimensional space---an object that can be traversed by the Simplex Algorithm defined in this paper. Such polytopes may hopefully enjoy the simplex method efficiency enjoyed by finite-dimensional linear programs. This offers a kind of ``optimizer's perspective'' on the geometric question of defining a polytope---a polytope is something that we have an algorithmic procedure to traverse extreme points to optimize a linear functional. This definition includes the Hilbert cube, but also many others objects that satisfy \cref{ass:compact,ass:support-conditions,ass:constraints-bounded,ass:countably-many-constraints,ass:partial-sums-compact,ass:steepest-exists,ass:edge-lengths-bounded-away-from-zero,ass:bounded-at-extremes}. 

These assumptions recover key geometric properties that are essential for the simplex viewpoint, as we show in the following result. We will say two extreme points $p$ and $p'$ of $\mathcal P$ are \emph{connected by a path} if there exists a finite or infinite sequence of adjacent edges that define a path of extreme points $(p, p^1, p^2, \dots)$ such that $p^n$ converges to $p'$ in the limit. We say $\mathcal P$ is \emph{path-connected} if all of its extreme points are connected by a path.

\begin{proposition}
If a $\mathcal P$ defined in \cref{eq:region-of-interest} satisfies \cref{ass:compact,ass:support-conditions,ass:constraints-bounded,ass:countably-many-constraints,ass:partial-sums-compact,ass:steepest-exists,ass:edge-lengths-bounded-away-from-zero,ass:bounded-at-extremes} then it is path-connected and all its extreme points are exposed. 
\end{proposition}
\begin{proof}
We first show all extreme points are exposed. Fix an extreme point $e \in E$ and we show that the linear functional
\[
K_e (x) = \sum_{i\in A(e)}s_i(x) (1/2)^i,
\]
where $s_i(x)=b_i-\phi_i(x)$, exposes the point $p$. 

First, we claim that $K_e$ is a continuous linear functional for every $e\in E$. First, recall that $\xi = \sup_{i\in A, e\in E} s_i(e)< \infty$ by \cref{ass:bounded-at-extremes} and note that $s_i$ are continuous convex functions. By the Bauer Maximum Principle \cite[Theorem~7.69]{hitchhiker} a continuous convex function on a compact convex set attains its maximum at an extreme point, so we conclude that $\sup_{i\in A, p\in P} s_i(p)\le \xi <\infty$. Second, the Weierstrass M-Test provides that a series of functions converges uniformly whenever its terms are bounded by terms of a second series which is absolutely convergent. We have $s_i(p)(1/2)^i \le \xi (1/2)^i$ for all $i \in A$ and $
\sum_{i\in A(e)} \xi (1/2/)^i \le 2\xi  < \infty$
so that the series  $\sum_{i\in A(e)} s_i(p)(1/2)^i$ is uniformly convergent. Finally, by the Uniform Limit Theorem, the limit of a uniformly convergent sequence of continuous functions is continuous.

Second, we show that $K_e$ is uniquely optimized at $e$ over $\mathcal P$. Note that $s_i(p) \ge 0$ for all $i\in A(e)$ since $p\in \mathcal P$. Moreover $s_i(p) = 0$ for all $i\in A(e)$ implies $p=e$ since then $0=b_i - \phi_i(p)$ or $\phi_i(p)=b_i$ for all $i\in A(e)$ so that  $p=e$ uniquely since $\mathcal H(e) =\{e\}$ by \cref{prop:extreme-then-intersect}. Since $s_i(p) = 0$ for all $i\in A(e)$ is optimal for the problem 
\begin{equation}\label{eq:last-problem}
\min_{p\in \mathcal P} K_e(p),
\end{equation}
extreme point  $e$ is the  unique optimal solution for $\min_{p\in \mathcal P} K_e(p)$ so that
$\min_{p\in \mathcal P} K_e(p)= K_e(e)$. Hence all extreme points are exposed points. 

Finally, we show every pair of extreme points is connected by a path of edges. The simplex method applied to the minimization   problem \cref{eq:last-problem} will by \cref{cor:convergence-in-solution} generate a sequence of edge connected extreme points $p^n$ terminating in $e$ or converging to $e$ as $n\rightarrow \infty$ regardless of starting extreme point $e^\prime$ since \cref{eq:last-problem} has the unique solution $e$. We conclude that there is an edge path connecting any two extreme points $e$ and $e^\prime$ in polytope $\mathcal P$.
\end{proof}

\section{Conclusion with a discussion of the assumptions}\label{s:conclusions}

In this paper we have studied a geometric version of the simplex method for topological vector spaces. We imposed a collection of assumptions to ensure that the geometric mechanics of the simplex method---moving from extreme points to extreme points along ``improving'' edges---converges in objective value to optimality. We show that our conditions generalize those found in the literature and cover the Hilbert cube, an interesting example that has troubled earlier definitions of the notion of a polytope (even though it is arguably the simplest possible feasible region of a linear program in our context).

A natural concern with any infinite-dimensional extension of the simplex method is the number and strength of structural assumptions required to recover the familiar geometric picture from finite dimensions. We emphasize that the assumptions introduced in this paper are not designed to be minimal. Rather, they serve to isolate the precise points at which infinite dimensionality obstructs the classical simplex argument and to make those obstructions explicit. In finite-dimensional linear programming, properties such as isolation of extreme points, nondegenerate edges, finite basis representations, and assured attainment of steepest descent directions are guaranteed by finite-dimensional polyhedrality. In general locally convex spaces, these properties must be imposed directly.

Viewed in this light, our framework should be understood as a structural existence result: it identifies a broad class of closed convex sets—extending beyond previously studied settings—for which a geometric simplex method is well-defined in principle and converges in value. Several of the assumptions jointly rule out distinct forms of geometric or analytic degeneracy, and we do not claim their logical independence. However, they are consistent, since, for example, the Hilbert cube satisfies all assumptions. Still, a number of the assumptions imposed in this paper reflect genuine obstructions that arise when attempting to extend the geometric simplex method beyond finite dimensions, rather than mere technical conveniences. 

Some form of compactness or effective boundedness (\cref{ass:compact}) is unavoidable if one wishes to guarantee the existence of extreme points and attainment of optimal solutions for continuous linear objectives. In general locally convex spaces, closed convex sets may lack extreme points altogether or admit minimizing sequences with no convergent subsequence, rendering any simplex-type method ill-posed at the outset (see, e.g., \cite{anderson-nash}). 

Likewise, uniform separation of active and inactive constraints at extreme points (\cref{ass:support-conditions}) is necessary to exclude extreme points with vanishing slack, a phenomenon closely related to non-exposed extreme points studied by \cite{klee1956strict,phelps1969infinite}. In the presence of diminishing slack, extreme points fail to behave as isolated geometric vertices, and no meaningful notion of an edge direction along which to move can be defined.

Several further assumptions rule out failure modes that are specific to infinite dimensions and have no finite-dimensional analogue. Uniform boundedness of constraint functionals (\cref{ass:constraints-bounded,ass:bounded-at-extremes}) prevents pathological scaling in which constraint normals grow without bound along geometrically relevant directions, destabilizing the construction of edges and descent rates. Similarly, the requirement that edge lengths be bounded away from zero (\cref{ass:edge-lengths-bounded-away-from-zero}) excludes situations in which sequences of adjacent extreme points accumulate, allowing infinitely many pivots within a compact region with negligible objective improvement. Finally, uniform summability of edge costs (\cref{ass:uniform-convergence-edge-costs}) reflects the necessity of controlling how infinitely many local descent directions aggregate; without such control, infinitesimal but nonsummable improvements can defeat convergence in value. 

Taken together, these considerations suggest that while individual assumptions may admit reformulation or weakening, any geometric simplex method that retains the classical interpretation of moving along edges between adjacent extreme points must confront—and in some form exclude—the same underlying obstructions identified here. Determining which conditions can be weakened, replaced, or derived from more intrinsic geometric properties is an important direction for future work. We hope that by making these requirements explicit, the present work provides a useful organizing framework for understanding when simplex-type methods can meaningfully extend beyond finite dimensions.

Additionally, future work may extend this geometric perspective in several directions.  There is scope for a theory of degeneracy, cycling, and perturbation analysis in this broader setting. One might also investigate what further assumptions guarantee that each iteration of the Simplex Method can be executed in finite time.

\appendix

\section{Proofs of lemmas}

\subsection{Proof of \cref{lemma:0}}

Let line $\ell (p, r) \triangleq \{x\in X \mid x=p+\lambda r, \lambda \in \R \}$ where $r, p\in X$ and let hyperplane $H= \{x\in X \mid \phi(x)=b\}$ where $\phi$ is a linear functional on $X$ and $b\in R$. Suppose $\ell (p,r)$ intersects $H$, i.e., there is a $\lambda_0 \in \R$ such that $\phi(p+\lambda_0 r) = b$. Suppose further that $\ell (p,r)$ does not lie in $H$. Then $\phi(r) \neq 0$ for otherwise $b=\phi(p+\lambda_0 r)=\phi(p)+\lambda_0 \phi(r)=\phi(p)$  in which case  $\phi(p+\lambda r)=\phi(p) + \lambda \phi(r)=b$ for all $\lambda \in \R$ and $\ell(p,r)$ would lie in $H$. Now suppose $\ell(p,r)$ meets $H$ in more than one point. Then there are $\lambda_1$ and  $\lambda_2$, $\lambda_1\neq \lambda_2$ such that $\phi(p + \lambda_1 r)=\phi(p+\lambda_2 r)$ so that $\phi(p)+\lambda_1\phi(r)=\phi(p)+\lambda_2 \phi(r)$ or $\lambda_1 \phi(r)=\lambda_2 \phi(r)$ or $\lambda_1/\lambda_2=1$ since $\phi(r)\neq 0$. Contradiction since $\lambda_1\neq \lambda_2$. Hence $\ell(p,r)$ must meet $H$ in exactly one point. \qed

\subsection{Proof of \cref{lemma:00}}

Let $S$ be a non-trivial affine subspace in $X$ and let $H$ be a hyperplane intersecting $S$ in a singleton point $p$. Let $q$ be another point in $S$ distinct from $p$ which we know exists since $S$ is non-trivial. Then the  line $\ell(p,q-p)=\{p+\lambda(q-p) \mid \lambda \in R\}\subseteq S$ since $S$ is an affine subspace. Suppose there is a point $o\in 
S$ with $o\notin \ell(p,q-p)$. Consider $\ell(p,q-p)$ translated to pass through $o$, i.e. $\ell(o,q-p)=\{o+\lambda (q-p) \mid\lambda\in \R\}.$ Note that $\ell(o,q-p)\subset S$ since it is a translate of a set in $S$. By \cref{lemma:0}, $\ell(o,q-p)$ intersects $H$ in a single point $p'$. 

\medskip Claim:  $\ell(o,q-p)\cap \ell(p,q-p) = \emptyset$. Proof: If not, there is a $y\in X$ with $y=o+\lambda'(q-p)=p+\lambda^{\prime\prime}(q-p)$ or $o-p=(\lambda' - \lambda^{\prime\prime})(q-p)$ in which case $o=p+(\lambda' - \lambda^{\prime\prime})(q-p)$ so that $o\in \ell(p,q-p)$. Contradiction. \medskip

Hence $p'\notin \ell(p,q-p)$ so that $p'\neq p$. Contradiction since $H$ intersects $S$ in the unique point $p$. Hence $S=\ell(p,q-p)$ is a line through $p$. \qed

\subsection{Proof of \cref{lemma:redef-recession-cone}}

Let $d\in \mathcal R(p)$ Then $b_\alpha\geq \phi_\alpha(p+\lambda d)= \phi_\alpha(p)+\lambda \phi_\alpha (d)=b_\alpha+ \lambda \phi_\alpha (d)$ for all $\lambda \geq 0 \mbox{ and } \alpha\in A(p)$ since $\phi_\alpha (p) = b_\alpha$ for $\alpha \in A(p)$. Hence $\phi_\alpha(d)\leq 0 \mbox{ for all } \alpha\in A(p)$.

For the reverse,  let $d\in \{d \in X \mid \phi_\alpha(d) \leq 0 \mbox{ for all } \alpha\in A(p)\}$. Then $\phi_\alpha (p+\lambda d) = \phi_\alpha (p) + \lambda \phi_\alpha(d) = b_\alpha + \lambda \phi (d) \leq b_\alpha \mbox{ for all }\lambda \ge 0 \mbox{ and } \alpha \in A(p)$. \qed

\subsection{Proof of \cref{lemma:extreme-rays}}

Note that $d\in \eta_{\alpha}(p) \Leftrightarrow d+p\in \ell_{\alpha}(p)\cap S_{\alpha} \Leftrightarrow d\in \ell_{\alpha}(p) -p$ and $d\in S_{\alpha}-p \Leftrightarrow d\in \cap_{{\alpha'}\in A(p),{\alpha'}\neq \alpha} H_{\alpha'}-p$ and $\phi_{\alpha}(d)\leq 0$
 $\Leftrightarrow\phi_{\alpha'}(d)=0$ for 
${\alpha'} \in A(p), {\alpha'} \neq \alpha$ and 
$\phi_{\alpha}(d)\leq 0$ since subtracting $p$ translates $p$ in $\ell_{\alpha}(p)$ and $S_{\alpha}$ to the origin. 

First we will show that $\eta_{\alpha}(p),  \alpha\in A(p)$  is an extreme ray of $\mathcal R(p)$ Let $x\in \eta_{\alpha}(p), y\in \mathcal R(p)$ and $x-y\in \mathcal R(p)$. Then a) $\phi_{\alpha'} (x)=0$ for ${\alpha'}\neq \alpha$, b) $\phi_{\alpha}(x)\leq 0$, c) $\phi_{\alpha'}(y)\leq0$ and $\phi_{\alpha'}(x-y)\leq 0 \mbox{ for all }\alpha'\in A(p)$. From c), we have d) $\phi_{\alpha'}(x)\leq \phi_{\alpha'}(y)  \mbox{ for all }\alpha'\in A(p)$. From a) and d), we have e) $\phi_{\alpha'}(y)=0$ for all ${\alpha'} \neq\alpha$. 
From c) and e), we conclude $y\in  \eta_{\alpha}(p)$ and hence  $\eta_{\alpha}(p)$ is an extreme ray of $\mathcal R(p)$.

We turn now to showing that every extreme ray of $\mathcal R(p)$ is equal to $\eta_{\alpha}(p)$ for some $\alpha\in A(p)$. Let $\eta$ be an extreme ray of $\mathcal R(p)$ and suppose $\eta \ne  \eta_{\alpha}(p)$ for any $\alpha\in A(p)$. Let $d\in r$. Since $d\in \mathcal R(p)$ and $d\neq 0, \phi_{\alpha'} (d)\le 0$ for all $\alpha\in A(p)$ while $\phi_{\alpha_1}< 0$ for some $\alpha_1\in A(p)$ since $d=0$ uniquely satisfies $\phi_{\alpha'}(d)=0, {\alpha'} \in A(p)$. Also since $\eta\ne  \eta_{\alpha}(p)$ for any $\alpha \in A(p), \phi_{\alpha_2}(d)<0$ for some $\alpha_2\ne \alpha_1$ in $A(p)$. Let $d' \in \delta  \eta_{\alpha_1}(p) + (1-\delta)  \eta_{\alpha_2}(p)$ and choose $\delta\in (0,1)$ 
so that $d'\notin \eta$ and $d' \ne 0$. Note that $d'\in \mathcal R(p)$ by convexity of $\mathcal R(p)$. 
We want to choose $\lambda \in (0,1)$ so that 
\begin{equation}\label{eq:d-thing}
d-\lambda d' \in \mathcal R(p)  
\end{equation}
since then we have $d\in \eta, d-\lambda d' \in \mathcal R(p)$, and $d-(d-\lambda d')=\lambda d' \in \mathcal R(p)$ while $d-\lambda d' \notin \eta$ so that $\eta$ is not an extreme ray.

We have \cref{eq:d-thing} holds if $\phi_{\alpha'}(d-\lambda d')\le0$ for all ${\alpha'}\in \mathcal R(p)$. Now $d' \in \delta  \eta_{\alpha_1}(p) + (1-\delta)  \eta_{\alpha_2}(p)$ means $d' = \delta  d_1' + (1-\delta)  d_2'$ where $d_1'\in \eta_{\alpha_1}(p)$ and $d_2'\in \eta_{\alpha_2}(p)$. Then $\phi_{\alpha'} (d-\lambda d' )= \phi _{\alpha'}(d)-\lambda (\delta \phi_{\alpha'}(d_1') - (1- \delta) \phi_{\alpha'} (d_2'))$. 

We have for ${\alpha'}=\alpha_1, \phi_{\alpha'}(d)<0, \phi_{\alpha'} (d_1')<0, \phi_{\alpha'} (d_2')=0$,
for ${\alpha'}=\alpha_2, \phi_{\alpha'} (d)<0, \phi_{\alpha'} (d_1')=0, \phi_{\alpha'} (d_2')<0$ and for ${\alpha'}\neq \alpha_1, {\alpha'} \ne \alpha_2, \phi_{\alpha'} (d)\le 0, \phi_{\alpha'} (d_1')=0, \phi_{\alpha'} (d_2')=0$. Hence \cref{eq:d-thing} holds for 
$\lambda < \min (  {\phi_{\alpha_1}(d)}/ {\delta \phi_{\alpha_1}(d_1')},  {\phi_{\alpha_2}(d)} /{(1-\delta) \phi_{\alpha_2}(d_2')})$. Therefore, $\eta$ is not an extreme ray. Contradiction. \qed

\subsection{Proof of \cref{lemma:x-sub-k-get-more-feasible}}

Let $x^n(0) = \sum_{k=1}^n x_k(0)$. Let $i \in \{1,2,\ldots,n\}$. Note that for $i\ne k$,  
\[\phi_i (x_k(0);0)= (\phi_k(x;0)/\phi_k(q_k(0);0)) \phi_i(q_k(0);0) = (\phi_k(x;0)/\phi_k(q_k(0);0))\phi_i(q_k(0);0)=0\] since $\phi_i(q_k(0);0) = 0$. Hence $\phi_i (x^n(0);0) = \phi_i (\sum_{k=1}^n x_k(0);0) =\phi_i(x_i(0);0)=\phi_i(x;0)$ since $r_i(0) \in H_i (x;0)$. Hence $x^n(0)\in H_i(x;0)$ for all $i = 1,2,\ldots,n$. \qed

\subsection{Proof of \cref{lemma:technical-convergence-lemma}}

Let $z=\lim_{n\rightarrow \infty} y_n$. Suppose $z \notin H_k(y;0)$ for some $k$. Put an open set $U$ around $z$ not intersecting closed set $H_k (y;0)$. Let $N$ be large enough so that $y_n \in U$ for all $n \ge N$. Set $M=\max \{k, N\}$. Then $y_M\in U$ and therefore $y_M \nsubseteq H_k(y;0)$ while $y_M \in \cap_{n=1}^M H_k(y;0)$. Contradiction. Hence $z \in H_k(y;0)$ for all $k=1,2,\ldots$. 
Then $z\in \cap_{k=1}^\infty H_k(y;0) =\{y\}$ so that $z=y$, noting that $\cap_{k=1}^\infty H_k(y;0) = \{y\}$ since $\cap_{k=1}^\infty H_k(0) = \{0\}$. \qed

\subsection{Proof of \cref{lemma:termination-condition}}

Suppose the simplex algorithm does not terminate at any finite iteration $n$. Then $\gamma(p^n)<0$ for all $n$. Suppose $\gamma(p^n) \not \rightarrow 0$ as $n\rightarrow\infty$. Then there is an $\epsilon>0$ and a subsequence $n_1,n_2, \ldots$ such that $\gamma(p^{n_k})<-\epsilon$ for $k=1,2,\ldots$. 
But then $c(p^{n_k+1})\le c(p)-k\epsilon\nu$ for all $k$ and therefore $c(p^{n_k+1})<c(x^*)$ for $k>(c(p)-c(x^*))/\epsilon\nu$. Contradiction since $x^*$ is a least cost point in $\mathcal P$. \qed

\section{Technical issue in \cite{ghate2021simplex} and proposed fix}\label{s:fix-ghate}

There is a technical issue with the results in \cite{ghate2021simplex}. In particular, no CILP defined in \cite{ghate2021simplex} satisfies both of that paper's assumptions \cite[(A3)]{ghate2021simplex} and \cite[(A5)]{ghate2021simplex}. To see this, note that any doubly-infinite matrix $A$ satisfying \cite[(A3)]{ghate2021simplex} is a Hilbert-Schmidt operator, since the sums of the squares of the entries of any such $A$ is finite.%
\footnote{To see that an operator $A$ satisfying \cite[(A3)]{ghate2021simplex} is a Hilbert-Schmidt operator, let $X$ be the Hilbert space of sequences $x=(x_j)_{j\ge1}$ with norm 
$\|x\|_X^2=\sum_{j\ge1}\delta^{2j}|x_j|^2$, and let $Y$ be defined analogously with weight $\beta^i$ on the $i$th coordinate, where $0<\beta,\delta<1$. 
Define the isometric isomorphisms $T_\delta:X\to\ell^2$ and $T_\beta:Y\to\ell^2$ by 
$(T_\delta x)_j=\delta^j x_j$ and $(T_\beta y)_i=\beta^i y_i$. 
Conjugating the operator $A=(a_{ij})$ yields $A^0:=T_\beta A T_\delta^{-1}:\ell^2\to\ell^2$ with matrix entries
$m_{ij}=\beta^i a_{ij}/\delta^j$. 
Under \cite[(A3)]{ghate2021simplex}, there exist $\bar a>0$ and $\alpha\in(0,\delta)$ such that 
$|a_{ij}|\le \bar a\,\alpha^j$ for all $i,j$, and hence $|m_{ij}| \le \bar a\,\beta^i(\alpha/\delta)^j$. It follows that
\[
\sum_{i=1}^\infty\sum_{j=1}^\infty |m_{ij}|^2
\le \bar a^2
\Big(\sum_{i=1}^\infty \beta^{2i}\Big)
\Big(\sum_{j=1}^\infty (\alpha/\delta)^{2j}\Big)
<\infty,
\]
since $\beta<1$ and $\alpha<\delta$. 
Therefore, $A^0$ is a Hilbert--Schmidt operator on $\ell^2$, and by invariance of the Hilbert--Schmidt class under conjugation by isometries, $A$ is Hilbert--Schmidt as an operator from $X$ to $Y$.}
It is well known that a Hilbert-Schmidt operator is compact  and no compact operator can be onto between two infinite-dimensional spaces (see, for instance, \cite[Theorem 4.25]{RudinFA1991} and its surrounding discussion), thus violating \cite[(A5)]{ghate2021simplex}.

There is a challenge in weakening \cite[(A3)]{ghate2021simplex} to recover \cite[(A5)]{ghate2021simplex}: it can easily compromise \cite[Lemma~8.2]{ghate2021simplex}, which argues that convergence of bases in the product discrete topology implies convergence in the operator norm. In order to recover \cite[(A5)]{ghate2021simplex} by weakening \cite[(A3)]{ghate2021simplex}, one needs to strengthen the conditions leading up to \cite[Lemma~8.2]{ghate2021simplex}. One combination that works is the following:

\begin{enumerate}[label=(A3$^\prime$)]
\item The doubly infinite matrix $A$ defines a bounded (continuous) linear operator from the Hilbert space $H$ to the Hilbert space $Y$.\label{ass:A3p}
\end{enumerate}
\begin{enumerate}[label=(A8$^\prime$)]
\item The set of bases $\mathcal{B}$ is closed in the product discrete topology (column-wise convergence). Furthermore, for any sequence of bases $\{B^n\} \subseteq \mathcal{B}$ converging to $B \in \mathcal{B}$ in this topology, the sequence of operators converges in the operator norm topology (i.e., $\|B^n - B\|_{op} \to 0$)\label{ass:A8'}
\end{enumerate}
\cref{ass:A3p} is a weakening of \cite[(A3)]{ghate2021simplex}, but \cref{ass:A8'} is a strengthening of \cite[(A8)]{ghate2021simplex}. One can check \cref{ass:A3p} is no longer inconsistent with \cite[(A5)]{ghate2021simplex} and \cref{ass:A8'} asserts the structure that the old \cite[(A3)]{ghate2021simplex} and \cite[(A8)]{ghate2021simplex} was leveraged to deliver  \cite[(Lemma 8.2)]{ghate2021simplex}. Fortunately, the motivating network flow example from \cite{ryan2018simplex} in \cite[Section 9]{ghate2021simplex} satisfies \cref{ass:A3p,ass:A8'}.

\subsubsection*{Declarations}

The authors have no conflicts of interest to declare.

\subsubsection*{Acknowledgements}

The authors would like to thank Google's Gemini, xAI's Grok, and OpenAI's ChatGPT for technical assistance in the preparation of this manuscript, including assistance in verifying proofs and identifying grammar issues and typos. We would also like to thank an anonymous associate editor at \emph{Mathematical Programming} for pointing out a technical issue with an earlier version of the manuscript. Christopher Thomas Ryan is funded by Natural Sciences and Engineering Research Council of Canada (NSERC) [Grant RGPIN-2020-06488], by the Social Sciences and Humanities Research Council of Canada (SSHRC) [Grant AWD-029333].

\small
\bibliographystyle{plainnat}
\bibliography{bibliography-hilbert}

\end{document}